\documentclass[12pt]{article}
\usepackage{thesis}
\usepackage{graphicx}
\usepackage{amssymb}
\usepackage{epstopdf}
\usepackage{comment}
\usepackage{amsmath}
\usepackage{amsthm}
\usepackage{stmaryrd}
\usepackage{amsxtra}
\usepackage{tikz}

\title{On the Classification of Algebras}
\author{Alex S. E. Levin}
\defensedate{March 28, 2013}
\thesis                             
\masterscience                      
\math                                 
\maygrad                            
\advisor{John Voight, Ph.D.}
\readerone{Jonathan Sands, Ph.D.}
\readertwo{David Dummit, Ph.D.}      
\chair{Christian Skalka, Ph.D.}

\newtheorem{thm}[equation]{Theorem}
\newtheorem*{thm*}{Theorem}
\newtheorem{prop}[equation]{Proposition}
\newtheorem{cor}[equation]{Corollary}

\newtheorem*{conj*}{Conjecture}

\newtheorem{lem}[equation]{Lemma}
\newtheorem{definition}[equation]{Definition}

\theoremstyle{remark}
\newtheorem{rem}[equation]{Remark}
\newtheorem{ex}[equation]{Example}

\numberwithin{equation}{section}

\begin{document}
\maketitle
\makeacceptance

\begin{abstract}
We classify (possibly noncommutative) algebras of low rank over a domain $R$. We first review results for algebras of rank 2 and for finite-dimensional division algebras over $\mathbb{R}$. These results motivate us to consider which algebras possess a standard involution. Our main result is that algebras of rank 3 are either commutative or possess a standard involution.
\end{abstract}

\begin{acknowledgements}  
Thanks to my supportive parents Barry and Mara. Comments and encouragement from my fellow Master's degree students have been invaluable, particularly Eric Roma, Mike Klug, Jared Krogsrud, Sam Schiavone, Brian Lang, Andy Reagan, and Jonathan Godbout. I could not have completed this work without the background and skills I developed in my coursework, especially in the Algebra and Real Analysis sequences taught by Professors John Voight and Jonathan Sands, respectively. Thank you to the members of the committee and to my advisor Professor John Voight, without whose guidance, tireless editing, and infectious energy this thesis would not exist. Thanks above all to my beloved wife, Andrea.
\end{acknowledgements}

\tableofcontents


\mainmatter

\section{Introduction}

In a vector space we enjoy the ability to combine vectors by addition and to scale by elements of a field. But we cannot, in general, multiply vectors. If we ``upgrade" our space to permit multiplication (subject to a few restrictions), we obtain an algebra. 
Generalizing from vector spaces to free modules over a domain, we make the following definition.

\begin{definition}
Let $R$ be a domain (with $1$). A \emph{free} $R$-\emph{algebra} is an associative ring $A$ (with $1$) equipped with an embedding of $R$ into the center of $A$ that is free as an $R$-module. The \emph{rank} of a free $R$-algebra is its rank as an $R$-module. We will assume that all algebras have finite rank. We identify $R$ with its image $R \hookrightarrow R  \cdot 1\subset A$.
\end{definition}

Algebras allow us to multiply tuples. For example, we can multiply pairs of real numbers in the quadratic (rank 2) $\mathbb{R}$-algebra $\mathbb{C}$, and we can multiply quadruples in the quaternion algebras of Sir William Rowan Hamilton.

In this thesis, we contribute to the classification of finite-rank algebras, up to isomorphism.
Algebras are classified first by rank. Much work has been done for rings of low rank \cite{Grossand, Gross, Voightlow}. There are different results for different ranks and choice of $R$.

The results for free quadratic (rank 2) algebras and finite-rank division algebras over $\mathbb{R}$ (see \cite{Palais} and sections 2, 3) suggest that possession of a standard involution is a useful feature with which to compare algebras.  

\begin{definition}\label{def:involution}An \emph{involution} $^-: A\to A$ is an $R$-linear map satisfying $\overline{\overline{x}} = x$ for all $x\in A$, $\overline{1} = 1$, and $\overline{xy} = \overline{y}\,  \overline{x}$ for all $x, y\in A$. An involution is \emph{standard} if it additionally satisfies $x\,  \overline{x}\in R$ for all $x\in A$.\end{definition}
 In section 4 we include some notes on standard involutions and the related topic of exceptional rings.

In section 5 our first result (Theorem \ref{thm:associative}) is a universal multiplication table for a cubic algebra over a domain $R$. Our main result, Theorem \ref{thm:main} in section 5, is that free cubic algebras (algebras of rank 3) are either commutative or possess a standard involution. The unique algebra in the intersection of these cases we call the \emph{nilproduct} ring, where the product of non-scalar basis elements is zero.

In section 6, we generalize the classification theorem of Frobenius from finite-rank division algebras over $\mathbb{R}$ to finite-rank division algebras over a field $F$ of characteristic not equal to 2 which possess a standard involution.
 We conclude with some discussion of future work.
 
\section{Quadratic Algebras}
In this section, we present some results in the classification of (associative) free quadratic algebras over a commutative ring $R$.
\subsection{Case 2 is a unit of $R$}
First suppose $R$ is a domain with $2\in R^\times$.
We show that the isomorphism classes of quadratic $R$-algebras are in bijection with the set of discriminant classes $d\in R/R^{\times 2}$. The results of this section are standard. For a reference, see Dummit and Foote \cite{dandf} or Voight \cite{Voightlow}.

Let $R$ be a commutative ring (with 1) and with $2\in R^\times$, and let $A$ be a free $R$-algebra $A$ of rank 2: we say that $A$ is \emph{free quadratic R-algebra}.  

\begin{lem}\label{lem:basis1}
There exists $x\in A$ such that $1, x$ is a basis for $A$. 
\end{lem}
\begin{proof}
Let $u, v$ be a basis for $A$. Then $1 = au+bv$, for some $a, b\in R$. Let $I = (a, b)\subset R$ be the ideal generated by $a, b$. Consider the ideal of $A$
\begin{align*}
IA &= \left\{ \sum_{i=0}^k r_i  (p_iu + q_iv): k\in \mathbb{N} , r_i\in I, p_iu + q_iv\in A\right\}.
  \end{align*}
Observe that $1 = au+bv \in I A$, so $IA = A$. We claim that $IA = Iu\oplus Iv$. An element $cu+dv\in Iu\oplus Iv$ is a sum of elements $cu, dv\in IA$. Thus $Iu\oplus Iv\subset IA$. Next for $w\in A$ consider $cw\in IA$. With respect to the basis $u, v$ we may write $cw = c(ru+sv) = (cr)u+ (cs)v\in Iu + Iv$. So $IA\subset Iu\oplus Iv.$

We have that $Iu\oplus Iv = A = Ru\oplus Rv$. Projecting in the first coordinate, we see that $Iu = Ru$. So for $ru\in Ru$ there exists $su\in Iu$ such that $ru=su$. Then $(r-s)u = 0$. Since $u\not=0$ and $R$ is a domain, we have $r=s\in I$, so $R\subset I$. Hence $R = I$.
Thus there exist $s, t\in R$ such that $at + (-s)b = 1$. So the change of basis
$$
\begin{bmatrix}
a& b\\
s& t
\end{bmatrix}
\begin{bmatrix}
u\\
v
\end{bmatrix}
=
\begin{bmatrix}
1\\ 
su + tv
\end{bmatrix}
$$
mapping $u, v\mapsto 1, su+tv = x$ is invertible since $\det [\begin{smallmatrix} a&b\\ c&d \end{smallmatrix}] = at-sb=1 $, and therefore $1, su + tv$ is a basis for $A$.
\end{proof}
Let $1, x$ be a basis for $A$ as in Lemma \ref{lem:basis1}.
Consider multiplication in $A$: for $ax+b$, $cx+d$ with $a, b, c, d\in R$, we have by distributivity and associativity in $A$, 
\begin{align*}
(a x+b)(c x+d) = a c x^2 + (a d+b c)x + b d.
\end{align*}
(Notice that multiplication is commutative in $A$, since $A$ is generated by $x$ as an $R$-algebra.) 
So to define multiplication in $A$, it is necessary and sufficient to specify $t, n\in R$ such that $x^2  = tx-n$. Consequently, we may write $A\cong \dfrac{R[x]}{(x^2 - tx + n)}$. Thus the set of quadratic algebras equipped with a basis is in bijection with the set of pairs $(t, n) \in R^2$.

Now we consider the isomorphism classes. Given $A = \dfrac{R[x]}{(x^2 - tx + n)}$, we define the \emph{discriminant} of $A$ by disc$(A) = t^2 - 4n\in R/R^{\times 2}.$
\begin{lem}\label{lem:completesquare}
If $2\in R^\times$, then $\displaystyle \frac{R[x]}{(x^2 - tx + n)}\cong \frac{R[y]}{(y^2 - d)} = R[\sqrt{d}]$, where $d = \,$\rm{disc}($A) = dR^{\times 2}$.
\end{lem}
\begin{proof}
Consider the ring homomorphism $x\overset{\varphi}{\longmapsto} \dfrac{y}{2} + \dfrac{t}{2}$, which is bijective since $2\in R^\times$. To show that relations are preserved compute
\begin{align*}
\varphi(x^2 - tx + n)  &= \left(\frac{y}{2} +\frac{t}{2} \right)^2 - t \left(\frac{y}{2} +\frac{t}{2} \right) + n\\
	&= \frac{y^2}{4} + \frac{ty}{2} + \frac{t^2}{4} - \frac{ty}{2} - \frac{t^2}{2} + n\\
	&= \frac{(t^2 - 4n)}{4} - \frac{t^2}{4}  + n = 0.
\end{align*}\end{proof} 
By Lemma \ref{lem:completesquare}, up to isomorphism, we may take $t=0$. We have ``completed the square."

\begin{lem}\label{lem:equaldisc}
If $A\cong B$ as free quadratic $R$-algebras, then \rm{disc}$(A) = $ \rm{disc}$(B)$ (as classes in $R/R^{\times 2}$).
\end{lem}
\begin{proof}
Express $A\cong\dfrac{R[x]}{(x^2 - tx + n)}$ with basis $1, x$ and $B\cong \dfrac{R[y]}{(y^2 - Ty + N)}$ with basis $1, y$, so that disc$(A) = t^2 - 4n$ and disc$(B) = T^2 - 4N$. By hypothesis there exists an isomorphism $\varphi :A\to B$ and $a\in R^\times$, $b\in R$ such that $\varphi(x) = ay+b$. 
Compute 
\begin{align*}
\varphi(x^2 - tx + n) = 0 &= (ay+b)^2  - t(ay+b) + n\\
		&= a^2y^2 + 2aby + b^2 - aty - bt + n\\
		&= a^2(Ty - N) + 2aby + b^2 - aty - bt + n\\
		&= (a^2T + 2ab - at)y - a^2N+ b^2 - bt + n.
\end{align*}
Since $B$ is free, collecting coefficients of $1, y$ and using the fact that $a\in R^\times$, we get $aT = t - 2b$ and $a^2N = b^2 - bt + n$. Therefore 
$a^2(T^2 - 4N) = t^2 - 4n$, which shows that disc$(A) = \,$ disc$(B)$.
\end{proof}

\begin{thm}\label{thm:congquadratics}
Let $2\in R^\times$ and $A = \dfrac{R[x]}{(x^2 - d)}$ and $B = \dfrac{R[y]}{(y^2 - D)}$ be quadratic algebras. Then $A\cong B$ as algebras if and only if $dR^{\times 2} = DR^{\times 2}$.
\end{thm}
\begin{proof}
Suppose $dR^{\times 2} = DR^{\times 2}$. Then $d = a^2D$ for some $a\in R^\times$. Consider the $R$-linear map $\varphi :R[x]\to B$ given by $\varphi(1)= 1$ and $\varphi(x)=ay$.
Let $px+q$ and $rx+s$ be elements of $A$. Then
\begin{align*}
\varphi((px+q)(rx+s)) &= \varphi(prx^2 + (ps+qr)x + qs) = prd + qs + a(ps+qr)y\\
&= pr(a^2D) + qs + a(ps + qr)y = a^2pry^2 + a(qr+ps)y + qs\\
&= (apy+q)(ary+s) = \varphi(px+q)\varphi(rx+s).
\end{align*}
Together with linearity this shows that $\varphi$ is an algebra homomorphism.
We compute the kernel of $\varphi$:
\begin{align*}
\varphi(x^2 - d) &= \varphi(x^2) - d = \varphi(x)^2 - d  = a^2y^2 - a^2D=  
a^2(y^2 - D) = 0,
\end{align*}
which shows $(x^2 - d)\subset \ker(\varphi)$. Next suppose $\varphi(p+qx) = 0$ for $p+qx\in A$. Then $p + aqy = 0$ so $p = -aqy$. Since $y$ is a basis element, $aqy\in R$ if and only if $aq = 0$, so $p, q = 0$. Hence $(x^2-d) = \ker(\varphi)$. Since $a\in R^\times$, $\varphi$ is surjective. Therefore by the first isomorphism theorem, $A\cong B$.

Together with Lemma \ref{lem:equaldisc} this completes the proof.
\end{proof}

\begin{cor}\label{cor:quadraticsbij}
If $2\in R^\times$, then there is a bijection
\begin{align*}
\left\{ \begin{array}{c} \text{quadratic } R\text{-algebras} \\ \text{up to isomorphism} \end{array} \right\} &\longrightarrow R/R^{\times 2}\\
[A]&\longmapsto \rm{disc}(A)\\
\left[\left.\raisebox{.2em}{$R[x]$}\middle/\raisebox{-.2em}{$(x^2-d)$}\right.\right] &\longmapsfrom dR^{\times 2}
\end{align*}
\end{cor}
\begin{proof}
By Theorem \ref{thm:congquadratics}, the isomorphism classes of quadratic $R$-algebras correspond to discriminants in $R/R^{\times 2}$.
\end{proof}
\begin{ex}\label{ex:disc0} Consider the class of discriminant 0. This class is represented by the algebra with $t = n = 0$, namely $\displaystyle A\cong \frac{R[x]}{(x^2)}$, and $x$ is nilpotent with $x^2=0$.\end{ex}

\begin{ex}\label{ex:disc1} A representative of the class of discriminant 1 is the algebra with $t=0,\, n=1$, so that $x^2-x=0$,  and  $\displaystyle A\cong \frac{R[x]}{(x^2-x)}$. In this case, $x$ is idempotent. Since $x^2-x$ factors as $x$ and $x-1$, by the Chinese Remainder Theorem, we have $A\displaystyle \cong \frac{R[x]}{(x-1)}\times \frac{R[x]}{(x)}\cong R\times R$. 
 \begin{center}
\begin{tikzpicture}
\node at (-2.5, 3.6) {$\dfrac{R[x]}{(x^2 - x)}$};
\node at (-2.5, 2.8) {$ ax+b \mod x^2 - x $};
\node at (-2.5, 2.2) {$(r-s)x + s$};
\node at (3, 3.6) {$\dfrac{R[x]}{(x-1)}\times \dfrac{R[x]}{(x)}$};
\node at (4.5, 2.8) {$(ax+b \mod x-1, \ ax+b\mod x)$};
\node at (0, 0) {$\begin{array}{c}R\times R\\ (r, s) \mapsfrom (a+b, b) \end{array}$};
\node at (.2, 3.8) {$\sim$};
\node at (.2, 3.4) {\footnotesize{C.R.T.}};
\draw [<-] (.2, .8) -- (1.8,2.4); 
\draw [->] (-1, 3.6) -- (1.4, 3.6);
\draw [->] (-.2, .8) -- (-1.8, 1.6);
\end{tikzpicture}
\end{center}
The induced multiplication is componentwise.  
 Multiplication in this class is given by $(a x+b)(c x+d) =  (ac + a d+b c)x + b d$. This class of algebras also has representative $\dfrac{R[x]}{(x^2 - 1)}$.
 \end{ex}

\begin{ex}\label{ex:RisZ}
If we remove the requirement that $2\in R^\times$, we sometimes achieve the same classification. The following was known to Gauss \cite{Voightlow}. For example, take $R = \mathbb{Z}$. Let $A = \dfrac{\mathbb{Z}[x]}{(x^2 - tx + n)}$, which has discriminant $t^2 - 4n$. Suppose $B  = \dfrac{\mathbb{Z}[y]}{(y^2 - Ty + N)}$ is isomorphic to $A$ via $\varphi$ given by $1\mapsto 1$, $x\mapsto y+k$ for some $k\in \mathbb{Z}$. Then
\begin{align*}
\varphi(x^2 - tx + n) = 0 &= (y+k)^2 - t(y+k) + n\\
			&= y^2 + 2ky + k^2 - ty - tk + n\\
			&= (T-(t-2k))y - N + k^2  - tk+n.
\end{align*}
Thus $T = t-2k$ and $N = k^2 - tk + n$. Computing the discriminant of $B$ we get $T^2 - 4N = t^2 - 4n$. Letting $x\mapsto -y+k$ leads to identical results. Since $\mathbb{Z}^\times = \{ \pm 1 \}$, this covers all possible invertible maps.

Suppose conversely that $A$ and $B$ have equal discriminants $t^2 - 4n = T^2 - 4N$ (since $(\mathbb{Z}^\times)^2 = \{1\}$, these are single-element classes). Then $(t-T)(t+T) = 4(n-N)$, 
from which we see that $t - T$ and $t+T$ are even and that $
N = \dfrac{T^2 - t^2}{4}+n.
$ Then consider the isomorphism $\varphi:A\to B$ given by $1\mapsto 1$, $x\mapsto y+\dfrac{t-T}{2}$. We check that relations are satisfied:
\begin{align*}
\varphi(x^2 - tx + n) = 0 &= \varphi(x)^2 - t\varphi(x) + n\\
			&= \left( y + \frac{t-T}{2} \right)^2 - t\left( y + \frac{t-T}{2} \right)+n\\
			&= y^2 + y(t-T) + \frac{1}{4}(t-T)^2 - ty - \frac{t}{2}(t-T) + n\\
			&= y^2 + [(t-T) - t]y + \frac{1}{4}(t^2 - 2tT+T^2) - \frac{1}{2}t^2+ \frac{1}{2}tT+n\\
			&= y^2 - Ty  + \frac{t^2}{4} - \frac{tT}{2}+ \frac{T^2}{4} - \frac{t^2}{2}+\frac{tT}{2}+n\\
			&= y^2 - Ty + \frac{T^2 - t^2}{4} + n\\
			&= y^2 - Ty + N,
\end{align*}
which shows that $\varphi$ induces a surjective homomorphism $A\to  B$. To check that $\varphi$ is injective let $ax+b, cx+d\in A$ satisfy $\varphi(ax+b) = \varphi(cx+d)$. Then
\begin{align*}
0&= \varphi(ax+b) = \varphi(cx+d) = (a-c)\varphi(x) + b-d\\
&= (a-c)y + (a-c)\frac{t-T}{2} + b-d,
\end{align*}
which shows that $ax+b = cx+d$. Thus $\varphi$ is an isomorphism $A\to B$.
\end{ex}

\subsection{Case $R$ is a field of characteristic 2}

If we relax the condition that $2\in R^\times$, then the classification becomes much more complicated. For example, when $F$ is a field of characteristic 2, then there exist nonisomorphic quadratic $F$-algebras with equal discriminants. 
Consider $A = \dfrac{\mathbb{F}_2[x]}{(x^2+x+1)}$ and $B = \dfrac{\mathbb{F}_2[y]}{(y^2 + y)}$. Note that disc$(A) = $ disc$(B) = 1$. Since $x^2 + x + 1$ is irreducible, $A$ is a field. Algebra $B$ has zero divisors $y$ and $y+1$, so $B$ is not a field. Thus $A\not\cong B$. 

Therefore we cannot in general classify quadratic algebras by discriminant when $2\notin R^\times$. However, when $2\notin R^\times$ we can classify separable quadratic algebras $A$ using the Artin-Schrier group in $R$.
\begin{definition}\label{def:separable}Let $f\in F[x]$ be a polynomial of degree $n$. Let $K = F(\alpha_1, \alpha_2, \ldots, \alpha_n)$, where $\alpha_i$ are the roots of $f$. We say $f$ is \emph{separable} if $\alpha_i$ are distinct in $K$. Define the $F$-algebra $A = \dfrac{F[x]}{(f(x))}$ to be \emph{separable} when $f$ is separable. \end{definition}
\begin{lem}\label{lem:char2separable}
For a field $F$ of characteristic $2$, polynomial $f(x) = x^2 - tx + n\in F[x]$ is separable if and only if $t\not=0$, and $\dfrac{F[x]}{(f(x))}\cong \dfrac{F[y]}{(y^2 + y + n)}$ for some $n\in F$.
\end{lem}
\begin{proof}
Let $(x-\alpha_1)(x-\alpha_2)$ be the factorization of $f$ over $F(\alpha_1, \alpha_2)$. By multiplying $(x-\alpha_1)(x-\alpha_2) = x^2 - (\alpha_1+\alpha_2)x + \alpha_1\alpha_2$, we see that $t = \alpha_1+\alpha_2$. Thus $t\not=0$, if and only if $\alpha_1, \alpha_2$ are distinct.

 Thus a separable quadratic polynomial over $F$ is of the form $f(x) = x^2 - tx + n$ where $t\in F^\times$. For separable $A = \dfrac{F[x]}{(x^2 - tx + n)}$, we have $x^2 - tx + n = 0$, so $t^{-2}x^2 - t^{-1}x + t^{-2}n = 0$. Then $A\cong \dfrac{F[y]}{(y^2 - y + t^{-2}n)}$ with isomorphism $x\mapsto t^{-1}x = y$. Therefore we may take any separable quadratic $F$-algebra to be of the form $\dfrac{F[x]}{(x^2 + x + n)}$.
  \end{proof}

Let $\wp(F) = \{ r\in F: r = r^2 \}$ be the \emph{Artin-Schreier group} of $F$, with operation addition. To confirm that this is a group, check that for $r, s\in \wp(F)$, we have $r+s\in \wp(F)$: 
$$
(r+s)^2 = r^2 + 2rs + s^2 = r^2 + s^2 = r+s.
$$
Define the class of $a$ in $F$ to be $a + \wp(F) = \{ x\in F: a-x\in \wp(F) \}$.
\begin{prop}\label{prop:artinbij}
For any field $F$ with char$(F) = 2$, there exists a bijection of sets
\begin{align*}
\frac{F}{\wp(F)} &\longrightarrow
\left\{
\begin{array}{c}
\text{separable quadratic algebras}\\
\text{over }F \text{ up to isomorphism}
\end{array}
\right\}  \\
a + \wp(F) &\longmapsto \frac{F[x]}{(x^2 + x + a)} 
\end{align*}
\end{prop}
\begin{proof}
For the class $a+\wp(F)\in F/\wp(F)$, consider the map $a+\wp(F)\mapsto \dfrac{F[x]}{(x^2+x+a)}$. First we show that the map is well-defined. Suppose $a+\wp(F) = b+\wp(F)$. Then $b = a+r+r^2$ for some $r\in F$. Since $\dfrac{F[x]}{(x^2 + x + a)}\cong \dfrac{F[x]}{(x^2 + x + a + r+ r^2)}$ by the isomorphism $x\mapsto x+r$, the map is well-defined. By Lemma \ref{lem:char2separable}, the map is surjective. Next we show that the map is injective. Suppose $a+\wp(F)\not=b+\wp(F)$, and that $a+\wp(F)\mapsto A =  \dfrac{F[x]}{(x^2 +x + a)}$, and $b+\wp(F)\mapsto B =  \dfrac{F[y]}{(y^2+y+b)}$ with $A\cong B$. Then there exists an isomorphism $\varphi:A\to B$ with $1\mapsto 1$ and $x\mapsto p+qy$ for some $p, q\in F$ with $q\in F^\times$. Then
\begin{align*}
\varphi(x^2 + x + a) = 0 &= \varphi(y)^2 + \varphi(y) + a\\
				&= (p+qy)^2 + p+qy + a\\
				&= p^2 + 2pqy + q^2y^2 + p + qy + a\\
				&= q^2(y+b) + qy + a + p + p^2\\
				&= (q+q^2)y + a + q^2b + p + p^2
\end{align*} Collecting like terms we see $q(q+1) = 0$. Since $q\not=0$, $q= 1$. Then $a + b + p + p^2 = 0$, which contradicts $a+\wp(F)\not= b+\wp(F)$. Therefore the map is injective, and so the sets are in bijection.
\end{proof}
\subsection{Remarks on Quadratic Algebras}

For $A = \dfrac{R[x]}{(x^2 - tx + n)}$, consider the $R$-linear map $^-:A\to A$ induced by 
$\overline{x} = t-x$. Observe $x\, \overline{x} = n\in R$, that $\overline{(ax+b)(cx+d)} = \overline{cx+d}\, \overline{ax+b}$, and  $\overline{\overline{x}} = x$ (also that $x+\overline{x}\in R$). Thus this map is a standard involution (Definition \ref{def:involution}). Therefore all quadratic algebras possess a standard involution. It is known that when $A$ is quadratic this involution is unique \cite{Voightlow}. 
\section{Frobenius's Theorem}

A \emph{division ring} is a ring in which each nonzero element is a unit. In 1877, Ferdinand Frobenius \cite{Palais} categorized the finite-dimensional division algebras over the real numbers as consisting of $\mathbb{R}, \mathbb{C}$, and the Hamiltonians $\mathbb{H}$. In this section, we give a modern elementary formulation of the proof. In later sections, we generalize from the real field to a domain and from a division algebra to one possessing a standard involution (so this argument serves as a partial model for our later investigations).

\begin{definition}\label{def:hamiltonians}
The \emph{Hamiltonians} $\mathbb{H}$ are the $\mathbb{R}$-algebra with basis $1, i, j, ij$ where multiplication is defined $i^2 = j^2 = -1$ and $ji = -ij$.\end{definition} 
\begin{rem}\label{rem:hamiltonianinvolution}
For $x = a + bi + cj + dij \in \mathbb{H}$, the map 
\begin{align*}
^-: \mathbb{H}&\longrightarrow\mathbb{H}\\
x&\longmapsto \overline{x} =  a - bi - cj - dij
\end{align*}
is a standard involution (Definition \ref{def:involution}). Define the \emph{trace} of $x$ to be tr$(x) = x + \overline{x} = 2a$. We will show (Corollary \ref{cor:divisioninvolution}) that $x\overline{x} = a^2 + b^2 + c^2 + d^2\in \mathbb{R}$. Define the \emph{norm} of $x$ to be nr$(x) = x\overline{x}$. Observe that $x$ satisfies $0 = T^2 - \rm{tr}$$(x)T + \rm{nr}$$(x)\in \mathbb{R}[T]$. Thus every $x\in \mathbb{H}$ satisfies a quadratic polynomial in $\mathbb{R}[T]$. \end{rem}

\begin{lem}\label{lem:irreducibledegree}
A real polynomial $p(x)\in \mathbb{R}[x]$ that is irreducible over $\mathbb{R}$ has degree $1$ or $2$.
\end{lem}

\begin{proof} By the Fundamental Theorem of Algebra, the polynomial $p(x)$ factors over $\mathbb{C}$ as $p(x) =\displaystyle \prod_{k=1}^n(x-\alpha_k) $, for some $n\in \mathbb{N}$, where $\alpha_k$ are in $\mathbb{C}$. Since $p(x)$ is irreducible over $\mathbb{R}$, either $p(x)$ is linear or all $\alpha_k\notin\mathbb{R}$ and thus occur in conjugate pairs $(p(\alpha) = p(\overline{\alpha}) = 0)$. So $p(x)$ may be written 
$$p(x) = \prod_{m}(x-\alpha_m)(x-\overline{\alpha_m})  = \prod_{m}\left(
x^2 - (\alpha_m + \overline{\alpha_m})x + \alpha_m\overline{\alpha_m}\right),$$
for some $m\in \mathbb{N}$.
Since $\alpha + \overline{\alpha_m}$ and $\alpha_m\overline{\alpha_m}$ are real, if the degree of $p(x)$ is greater than 2, we have a contradiction. Thus the degree of $p(x)$ is less than or equal to 2.
\end{proof}

\begin{lem}\label{lem:divisionminimal}
If $D$ is a finite-dimensional division ring containing $\mathbb{R}$ and $\alpha\in D\setminus \mathbb{R}$, then the minimal polynomial of $\alpha$ is irreducible over $\mathbb{R}$. 
\end{lem}
\begin{proof}
Let $p(x)\in \mathbb{R}[x]$ be the minimal polynomial of $\alpha$ over $\mathbb{R}$. Suppose $p(x) = f(x)g(x)$ with $f, g\in \mathbb{R}[x]$,  both of nonzero degree (less than $n$); then $p(\alpha) = f(\alpha)g(\alpha) = 0$. Then since $D$ is a division ring, $f(\alpha)=0$ or $g(\alpha)=0$ and so $\alpha$ satisfies a real polynomial of degree less than $n$, a contradiction. Therefore, $p(x)$ is irreducible.
\end{proof}

\begin{cor}\label{cor:fieldRorC}
A finite extension $K$ over $\mathbb{R}$ is either $\mathbb{R}$ or $\mathbb{C}$. 
\end{cor}

\begin{proof}
Assume $K\not=\mathbb{R}$ and let $\alpha\in K\setminus\mathbb{R}$. Let $p(x)$ be the minimal polynomial of $\alpha$ over $\mathbb{R}$. By Lemma \ref{lem:divisionminimal} we have that $p(x)$ is irreducible, and since $p(x)$ is of degree at least 2, by Lemma \ref{lem:irreducibledegree}, $p(x)$ is quadratic. Since $p(x)$ factors completely over $\mathbb{C}$, we have $\alpha\in \mathbb{C}$. Hence $K\subset \mathbb{C}$. Then since $[K:\mathbb{C}]\ge 2$, we have $K = \mathbb{C}$.
 \end{proof}

\begin{thm}[Frobenius]\label{thm:frobenius}  The finite-dimensional division $\mathbb{R}$-algebras are $\mathbb{R}$, $\mathbb{C}$, and $\mathbb{H}$.
\end{thm}

\begin{proof} 
Let $D$ be a finite-dimensional $\mathbb{R}$-algebra. We have $\dim_{\mathbb{R}} D = 1$ if and only if $D = \mathbb{R}$. We will show that if $\dim_{\mathbb{R}} D > 1$ then $\mathbb{C}\subset D$, and then that if $\mathbb{C}\subsetneq D$, then $D = \mathbb{H}$.

Suppose $\dim_{\mathbb{R}} D > 1$. Then there exists $j\in D\setminus\mathbb{R}$. Since $D$ is finite dimensional, it cannot be the case that $1, j, j^2, \ldots j^n, \ldots$ are 
$\mathbb{R}$-independent for all $n\in \mathbb{Z}_{\ge 0}$. Thus there exists $n\in \mathbb{Z}_{>0}$ such that  $1, j, j^2, \ldots, j^n$  are linearly dependent. So $j$ satisfies a real polynomial. Let $p(x)$ be the monic polynomial of smallest degree $n$ satisfied by $j$. 
Then the  $\mathbb{R}$-subalgebra of $D$ generated by $j$ is
$$
\mathbb{R}[j] = \{ a_0 + a_1 j + \cdots + a_{n-1} j^{n-1} : a_k\in \mathbb{R} \}\cong \frac{\mathbb{R}[x]}{(p(x))}.
$$ Since $\mathbb{R}$ is commutative and $j$ commutes with itself, $\mathbb{R}[j]$ is commutative.

\begin{lem}\label{lem:RjC} $\mathbb{R}[j]\cong \mathbb{C}$. \end{lem}
Since $D$ is a division ring, $p(x)$ is irreducible by Lemma \ref{lem:divisionminimal}. Then the ideal $(p(x))$ is maximal. Therefore $\displaystyle \mathbb{R}[j] \cong \frac{\mathbb{R}[x]}{(p(x))}$ is a field. Since $\mathbb{R}[j]$ is a field containing $\mathbb{R}$ and $j\notin \mathbb{R}$, by Corollary \ref{cor:fieldRorC} we have $\mathbb{R}[j]\cong \mathbb{C}$.\end{proof}

So if $\dim_\mathbb{R} D >1$ then $\mathbb{C}\subset D$, and 
$D = \mathbb{C}$ if and only if $\dim_\mathbb{R} D = 2$.

Next suppose that $\dim_\mathbb{R} D >2 $. Then there exists $i\in D$ where $i$ is $\mathbb{R}$-independent of $1$ and $j$. By the argument in Lemma \ref{lem:RjC}, we have $\mathbb{R}[i] \cong \mathbb{C}$.

\begin{lem}\label{lem:isquared}
If {\rm dim}$_\mathbb{R}D> 2$ then there exist $i, j\in D$ such that $1, i, j$ are independent and $i^2 = -1$.
\end{lem}
\begin{proof}
Since dim$_\mathbb{R}D> 2$ there exist $1, k, j$ $\mathbb{R}$-independent elements of $D$. Since the minimal polynomial of $k$ is quadratic, $k^2 = tk-n$ for some $t, n\in \mathbb{R}$. Replace $k$ with $k^\prime = 2k-t$. Then $1, k^\prime, j$ are independent and $(k^\prime)^2 = t^2 - 4n$. Note that $t^2 - 4n<0$ because the minimal polynomial $x^2 - (t^2 - 4n)$ of $k^\prime$ is irreducible. Finally, replace $k^\prime$ with $i = ak^\prime$ with $a\in \mathbb{R}_{>0}$ chosen so that $i^2 = -1$.
\end{proof}

Consider the map given by conjugation by $i$ in $D$: \begin{align*}
\varphi :D & \to D
\\ x & \mapsto ixi^{-1} = -ixi
\end{align*}
A check of the axioms shows that $\varphi$ is an $\mathbb{R}$-algebra isomorphism. Observe that $\varphi\circ\varphi = $ id. Since $\varphi^2 = 1$ in End$_D$, $\varphi$ satisfies the polynomial $T^2 - 1$, which has roots $\pm 1$. Thus the minimal polynomial of $\varphi$ idivides $T^2 - 1 = (T+1)(T-1)$. So the roots of the minimal polynomial of $\varphi$ are distinct, and then $\varphi$ is diagonalizable.

\begin{prop}\label{prop:Deigenspaces} Let $D^-$ be the eigenspace of eigenvalue $-1$ and $D^+$ be the eigenspace of eigenvalue $1$ with respect to the above map $\varphi$, as vector spaces over $\mathbb{R}$. Then $D = D^+\oplus D^-$. \end{prop}
\begin{proof}
Since $\varphi$ is diagonalizable, $D$ has a basis of eigenvectors of $\varphi$. Consider the bijection
\begin{align*}
\psi: D&\longrightarrow D^+\oplus D^-\\
	x&\longmapsto \left(\frac{x+ \varphi(x)}{2},  \frac{x- \varphi(x)}{2} \right)\\
	y+z&\longmapsfrom (y, z)
\end{align*}
The existence of $\psi$ shows that any basis of $D^+$, union a basis of $D^-$ is a basis of eigenvectors of $D$. So $D = D^+\oplus D^-$.\end{proof}

Note that since $\mathbb{R}\cup \{i \}\subset D^+$, we have $\mathbb{R}[i]\cong \mathbb{C}\subset D^+$.

\begin{prop}\label{prop:DplusC}
$D^+ = \mathbb{R}[i] = \mathbb{C}$.
\end{prop}
\begin{proof}
Let $x\in D^+$. Since $x\in D^+$,  $\varphi(x) = ixi^{-1} = x$. So $ix = xi$. Thus the $\mathbb{R}$-algebra generated by $\mathbb{R}[i]$ and $x$, $\mathbb{R}[i][x]\cong \mathbb{C}[x]$ is commutative. We check that $D^+$ is closed under multiplicative inverse:
\begin{align*}
x&=  ixi^{-1} = \varphi(x) \\
x^{-1}  &= ix^{-1}i^{-1} = \varphi(x^{-1})
\end{align*}

So $\varphi(x^{-1}) =x^{-1}$, which shows that $x^{-1}\in D^+$. Since $\mathbb{C}[x]$ is a commutative division algebra, it is a field.
Then by Corollary \ref{cor:fieldRorC}, $x\in \mathbb{C}$, so $D^+ = \mathbb{C}$.
\end{proof}

We have shown that $D = D^+$ if and only if $D = \mathbb{C}$. So now suppose that $\dim_\mathbb{R} D > 2$; then there exists $y\in D^-, y\not=0$. Then $iyi^{-1} = -y $ and so $iy = -yi$.  Note that $y\notin\mathbb{C}$. By the same reasoning as in Proposition \ref{prop:DplusC}, $y^{-1}\in D^-$. 
\begin{prop}\label{prop:DminusC}
$D^- = \mathbb{C}[y]$.
\end{prop}
\begin{proof}
Observe that for $y_1, y_2\in D^-$, $y_1 y_2\in D^+$:
\begin{align*}
\varphi(y_1y_2) = iy_1 y_2 i^{-1} = iy_1 (i^{-1} i) y_2 i^{-1} = (-y_1)( -y_2) = y_1y_2.
\end{align*}
Let $k\in D^-$. Note that $k = (k  y^{-1})y$. Since $k y^{-1}\in D^+ = \mathbb{C}$, we have $k\in \mathbb{C}y$. 
Conversely, if $k\in \mathbb{C}y$, then $k\in D^-$ since $y\in D^-$ and $\mathbb{C}\subset D^+$.
Therefore, $D^- = \mathbb{C}y$. \end{proof}

Therefore if dim$_\mathbb{R} D> 2$, then $D\cong D^+\oplus D^-\cong \mathbb{C}\oplus\mathbb{C}y$. As in Lemma \ref{lem:isquared} we may take $y^2  = -1$. So we have $D$ generated by $1, i, y, iy$ where $i^2 = y^2 = -1$ and $iy = -yi$. This describes $\mathbb{H}$.

\section{Standard Involutions and Exceptional Rings}

Algebras of rank 1, quadratic rings, and finite-dimensional division algebras containing $\mathbb{R}$ each possess a standard involution. Consequently, each element in such an algebra satisfies a polynomial of degree $\le 2$ with coefficients in that algebra's ring of scalars. In our classification, it is natural to consider which algebras possess such a map. In this section we discuss standard involutions, give a few important examples, and review some of their properties.

As in \S 1, let $R$ be a domain and let $B$ be an $R$-algebra.
\begin{definition}\label{def:involution}
An \emph{involution} $^-: B\to B$ is a self-inverse (i.e. $\overline{\overline{x}} = x$ for all $x\in A$) $R$-linear map satisfying $\overline{xy} = \overline{y}\,  \overline{x}$ for all $x, y\in B$. An involution is \emph{standard} if it additionally satisfies $x\,  \overline{x}\in R$ for all $x\in B$.
\end{definition}

\begin{lem}\label{lem:xplusxbar} If $^-$ is a standard involution on $B$, then $x + \overline{x}\in R$ for all $x\in B$. \end{lem}

\begin{proof} For $x\in B$, we have $(x+1)(\overline{x+1})\in R$, so 
 $x\, \overline{x} + x + \overline{x} + 1\in R$; since $x\overline{x}\in R$, subtracting $x\, \overline{x} + 1$ from each side we get $x + \overline{x}\in R$.
\end{proof}

Let $B$ have standard involution $^-$. For $x\in B,$ let $t(x)$ be the element of $R$ such that $x + \overline{x} = t(x)$. Then $\overline{x} = t(x) - x$, which shows that $^-$ is determined by the functional $t:B\to R$, $ x\mapsto t(x)$.

Observe that  $t$ is the sum of $^-$ and the identity, each of which is $R$-linear, and so $t$ is $R$-linear. Thus $t$ is an $R$-linear map. 

\begin{rem}\label{rem:invpoly} If $A$ has a standard involution then every $x\in A$ is a root of a polynomial $x^2 - x(x+\overline{x}) + x\, \overline{x}$ in $R[x]$.
\end{rem}

\begin{ex}\label{ex:quaternioninvolution}
A quaternion algebra over a field possesses a standard involution.
\end{ex}
\begin{proof}
Let $F$ be a field and $Q = \left( \dfrac{a, b}{F} \right)$ be a quaternion $F$-algebra with basis $1, i, j, ij$ satisfying $i^2 = a\in F^\times$, $j^2 = b\in F^\times$, and $ji = -ij$. Then consider the $F$-linear map $^-:Q\to Q$ given by $\overline{p + qi + rj + sij} = p - qi - rj - sij$. Clearly, $^-$ is self-inverse and fixes $F$. To check that $^-$ is an involution we show that $\overline{mn} = \overline{n}\, \overline{m}$ for all $m, n\in Q$. Let $m = p+qi+rj+sij$ and $n = w+xi+yj+zij$. Then
\begin{align*}
\overline{mn}&= \overline{pw + pxi + pyj+pzij} + \overline{qwi + aqx + qyij + aqzj}\\
	&\ \ \ +\overline{rwj  - rxij + bry - brzi} + \overline{swij  - asxj + bsyi  - absz}\\
	&= \overline{pw+aqx+bry-absz} +\overline{(px+qw-brz+bsy)i}\\
	&\ \ \ +\overline{(py+aqz+rw-asx)j} +\overline{(pz+qy-rx+sw)ij}\\
	&= pw+aqx+bry-absz -(px+qw-brz+bsy)i\\
	&\ \ \ -(py+aqz+rw-asx)j -(pz+qy-rx+sw)ij\\
\end{align*}
Next we compute $\overline{n}\, \overline{m}$:
\begin{align*}
\overline{n}\, \overline{m}&= (w-xi-yj-zij)(p-qi-rj-sij)\\
	&= pw  - qwi - rwj - swij   - pxi + aqx +rxij + asxj \\
	&\ \ \  - pyj -qyij +bry -bsyi - pzij -aqzj+brzi - absz\\
	&= pw+aqx+bry-absz -(px+qw-brz+bsy)i\\
	&\ \ \ -(py+aqz+rw-asx)j -(pz+qy-rx+sw)ij\\
\end{align*}
So $\overline{mn} = \overline{n}\, \overline{m}$, which shows that $^-$ is an involution. Next we check that $^-$ is standard:
\begin{align*}
m\overline{m}&= (p+qi+rj+sij)(p-qi-rj-sij)\\
	&= p^2 -pqi-prj-psij + pqi-aq^2-qrij - aqsj\\
	&\ \ \ +prj + qrij - br^2+brsi + psij + aqsj - brsi + abs^2\\
	&= p^2 - aq^2 -br^2 + abs^2\in F
\end{align*}\end{proof}

\begin{cor}\label{cor:divisioninvolution}
Finite-dimentional division algebras containing $\mathbb{R}$ possess a standard involution.
\end{cor}
\begin{proof}
On $\mathbb{R}$, the identity is a standard involution. The standard involution on $\mathbb{C}$ is complex conjugation. The Hamiltonians $\mathbb{H}$ are a quaternion algebra and have the standard involution described in Example \ref{ex:quaternioninvolution}.\end{proof}

\begin{ex}\label{ex:adjoint}
If $F$ is a field then the adjoint map $[\begin{smallmatrix} a& b\\c& d  \end{smallmatrix}]\mapsto [\begin{smallmatrix} d& -b\\-c& a  \end{smallmatrix}]$ is a standard involution on $M_2(F)$.
\end{ex}
\begin{proof}Consider the map that sends matrix $M$ to its adjoint:
\begin{align*}
^-: M_2(F)&\longrightarrow M_2(F)\\
\begin{bmatrix}
a& b\\
c& d
\end{bmatrix}
&\longmapsto 
\overline{\begin{bmatrix}
a& b\\
c& d
\end{bmatrix}}
=
\begin{bmatrix}
d& -b\\
-c&a
\end{bmatrix}
\end{align*}
To check that $^-$ is an involution, we first verify that $^-$ is self-inverse:
\begin{align*}
\overline{
\overline{
\begin{bmatrix}
a& b\\
c& d
\end{bmatrix}}
}&= \overline{
\begin{bmatrix}
d& -b\\
-c& a
\end{bmatrix}
} = 
\begin{bmatrix}
a& b\\
c& d
\end{bmatrix}
\end{align*}

Next we show that $^-$ is $F$-linear. Clearly, $^-$ fixes $[\begin{smallmatrix} 1&0\\ 0&1 \end{smallmatrix}]$. To verify that $^-$ is additive we compute:
\begin{align*}\overline{
\begin{bmatrix}a& b\\ c& d  \end{bmatrix} + 
\begin{bmatrix}p& q\\ r& s  \end{bmatrix}} &
=\overline{
\begin{bmatrix} a+p& b+q\\ c+r& d+s  \end{bmatrix}}=
\begin{bmatrix} d+s & -(b+q)\\ -(c+r) & a+p  \end{bmatrix}\\ 
	&= 
\begin{bmatrix} d& -b\\ -c& a \end{bmatrix} + 
 \begin{bmatrix} s& -q\\ -r& p \end{bmatrix}  = 
\overline{\begin{bmatrix}a& b\\ c& d  \end{bmatrix} }+ 
\overline{ \begin{bmatrix}p& q\\ r& s  \end{bmatrix} } 
\end{align*}
Then we show that $^-$ respects scaling by $k\in F$:
\begin{align*}
\overline{k\begin{bmatrix} a& b\\ c& d \end{bmatrix}}= \overline{ \begin{bmatrix} ka& kb\\kc&kd \end{bmatrix}} = 
\begin{bmatrix} kd&-kb\\ -kc&ka  \end{bmatrix}  
= k \begin{bmatrix}d &-b\\ -c&a  \end{bmatrix} =   k\ \overline{\begin{bmatrix} a& b\\c& d \end{bmatrix}}
\end{align*}
For $^-$ to be an involution it remains to verify that $\overline{xy} = \overline{y}\, \overline{x}$ for all $x, y\in M_2(D)$:
\begin{align*}
\overline{
\begin{bmatrix}
a& b\\
c& d
\end{bmatrix}
\begin{bmatrix}
p& q\\
r& s
\end{bmatrix}
} &= \overline{
\begin{bmatrix}
ap+br& aq+bs\\
cp+dr & cq + ds
\end{bmatrix}}
= 
 \begin{bmatrix}
 cq+ds & -(aq+bs)\\
 -(cp + dr)& ap+br
 \end{bmatrix}\\ 
 &=
 \begin{bmatrix}
s & -q \\
 -r & p
 \end{bmatrix}
 \begin{bmatrix}
 d & -b \\
 -c & a
 \end{bmatrix}
 = \overline{
\begin{bmatrix}
p& q\\
r& s
\end{bmatrix}} \ \overline{
\begin{bmatrix}
a& b\\
c& d
\end{bmatrix}}
\end{align*}
Therefore $^-$ is an involution. To determine whether $^-$ is standard we check that $x\overline{x}\in F$ for all $x\in M_2(D)$:
\begin{align*}
\begin{bmatrix} a& b\\ c& d  \end{bmatrix}
\overline{\begin{bmatrix} a& b\\ c& d  \end{bmatrix}  } &= 
\begin{bmatrix} a& b\\ c& d  \end{bmatrix}
\begin{bmatrix} d& -b\\ -c& a  \end{bmatrix} = 
\begin{bmatrix}  ad -bc& -ab + ba\\ cd - dc& -cb + da   \end{bmatrix}\\ &= (ad - bc)
\begin{bmatrix} 1& 0 \\ 0 & 1 \end{bmatrix}  \in F
\end{align*}
\end{proof}

\begin{ex}\label{ex:FcrossF}
On $F\times F$, the map $(x, y)\mapsto (y, x) = \overline{(x, y)}$ is a standard involution. 
This map is clearly $F$-linear and self-inverse. Since (component-wise) multiplication is commutative, $\overline{ab} = \overline{b}\, \overline{a}$ for all $a, b\in F\times F$. To see that this involution is standard we compute
\begin{align*}
(x, y)\overline{(x,y)} = (xy, yx) = xy(1, 1)\in F.
\end{align*}\end{ex}

An important class of rings possessing standard involutions are exceptional rings {\rm \cite{Grossand, Voightlow}}.
\begin{definition}\label{def:exceptional}
An \emph{exceptional ring}  is an $R$-algebra $B$ with the property that there is a left ideal $M\subset B$ with $B = R \oplus M$ such that the left multiplication map $M \to$ End$_R(M) $ factors through a linear map $t : M \to R$, i.e. $xy = t(x)y$ for $x, y \in M$. \end{definition}
\begin{prop}\label{prop:exceptionalinvolution}
The map $x \mapsto t(x)-x$ for $x \in M$ extends to a standard involution on $B$, (and so any exceptional ring possesses a standard involution).
\end{prop}
\begin{proof}
Let $^-:B\to B$ be the $R$-linear extension of the map $x\mapsto t(x)-x$,  fixing 1. To see that $^-$ is self-inverse compute
$$
\overline{\overline{x}} = \overline{t(x) - x} = t(x) - (t(x) - x) = x.
$$
Let $a, b\in R$ and $x, y\in M$. To check that $^-$ is an involution we compute
\begin{align*}
\overline{(a+x)(b+y)} &= \overline{ab+ay+bx + xy} = ab + a\overline{y} + b\overline{x} + t(x)\overline{y}\\
&= ab + at(y) - ay + bt(x) - bx + t(x)t(y) - t(x)y.
\end{align*}

Next we find the product
\begin{align*}
(\overline{b+y}) (\overline{a+x})&= (b + t(y) - y)(a + t(x) - x)\\
	&= ab + bt(x)  - bx + at(y) + t(x)t(y) - t(y)x - ay - t(x)y + t(y)x\\
	&= ab + at(y) - ay + bt(x) - bx + t(x)t(y) - t(x)y.
\end{align*}
Thus $^-$ is an involution.

To check that $^-$ is standard, compute
$$
x\overline{x} = xt(x) - x^2 = xt(x) - t(x)x = 0\in R.
$$\end{proof}

In \cite{Voightlow}, Voight shows that for algebras of rank 3, possession of a standard involution is equivalent to being an exceptional ring.
\begin{thm}[Voight] An $R$-algebra $B$ of rank 3 has a standard involution if and only
if it is an exceptional ring {\rm \cite{Voightlow}}. 
\end{thm}
We give an elementary proof of this theorem in the following section.

\section{Associative Algebras of Rank 3}

Let $R$ be a domain. In this section we classify $R$-algebras that are free of rank 3 as an $R$-module. First we produce a universal multiplication table for such algebras (Equation \ref{eqn:mtable}). Then our main result is as follows. 
\begin{thm}\label{thm:main} A free $R$-algebra of rank $3$ is either commutative or possesses a standard involution.
\end{thm}
We further show the only commutative algebra which also possesses a standard involution is the \emph{nilproduct algebra}, defined to be the algebra in which the product of non-scalar basis elements is zero: $A = \dfrac{R[x,y]}{(x^2, y^2, xy, yx)} = \dfrac{R[x, y]}{(x, y)^2}$.

In \cite{Voightlow}, Voight shows that $R$ is a direct summand of $A$, i.e., $A$ has a basis containing 1. Let $1, i, j$ be a basis of $A$. Then multiplication in $A$ is defined as follows, for some $a, b, c, d, e, f, l, m, n, x, y, z \in R$
\begin{equation}
\begin{array}{ccl}
i^2 &  = & a + bi + cj \\
ij &  = & d + ei + fj \\
ji &  = & l + mi + nj \\
j^2 &  = & x + yi + zj \\
\end{array}
\end{equation}

This multiplication table can be simplified by adjusting the basis. As in the paper by Gross and Lucianovic \cite{Grossand}, we define a \emph{good basis} to be a basis of the form $1, i, j$ where $ij\in R$. This adjustment is analogous to ``completing the square" in the quadratic case.
\begin{prop}\label{prop:goodbasis}
Every $R$-algebra of rank $3$ has a good basis.
\end{prop}
\begin{proof}
Following \cite{Grossand}, (which assumes that the algebra is commutative, though it is not necessary to do so) we apply the change of basis $1\mapsto 1, \ i \mapsto I = i-f, \ j\mapsto J = j-e$. Then $$IJ = (i-f)(j-e) = ij - ei - fj + ef = d+ef,$$ which is in the span of 1. Thus $A$ can be rewritten with the multiplication table:
\begin{equation}\label{goodbasis}
\begin{array}{ccl}
i^2 &  = & a + bi + cj \\
ij &  = & d \\
ji &  = & l + mi + nj \\
j^2 &  = & x + yi + zj \\
\end{array} 
\end{equation}with $a, b, c, d, l, m, n, x, y, z \in R$.\end{proof}

The requirement that $A$ be associative restricts the possible values of the coefficients in the above multiplication table.
For example, by associativity, we have the identity $(ii)j = i(ij)$. Applying the above multiplication table to the identity we obtain:
\begin{align*}
(ii)j &= i(ij)\\
(a + bi + cj)j& = i(d)\\
aj + bij + cj^2& =  di
\end{align*}
Subtracting $di$ from each side and applying the relation in table 
\ref{goodbasis} we get
$$
bd + cx+ (cy-d)i + (a + cz)j =  0,$$
and so $bd + cx  =  cy - d = a+cz = 0$. This allows the simplifications $d\mapsfrom cy$ and $a\mapsfrom -cz$.

The relations revealed by associativity are as follows:
\begin{enumerate}
\item From $(ii)j = i(ij)$ we have $bd + cx  =  d-cy = a+cz = 0$;
\item from $(ij)j = i(jj)$ we have $ay+dz = x+by = d-cy = 0$;
\item from $(ii)j = i(ij)$ we have $bd+cx = d-cy = a+cz = 0$;
\item from $(ij)i = i(ji)$ we have $am+dn = l+bm-d = cm = 0$;
\item from $(ji)j = j(ij)$ we have $dm+nx = ny = l-d+nz = 0$;
\item from $(ji)i = j(ii)$ we have $cx +bcy = bm-mn = bn-n^2 = 0$;
\item from $(jj)i = j(ji)$ we have $bm^2 - bmz - nx = x+by+mz-m^2-ny = mn-bm=0$.
\end{enumerate}
The relations in $1.$ allow the substitutions $d\mapsto cy$ and $a\mapsto -cz$. From $2.$ we may substitute $x\mapsto -by$. From $6.$ (and $d = cy$) we see that we may replace $l$ with $cy-nz$. Thus we have:
\begin{equation}
\begin{array}{ccl}
i^2 &  = & -cz + bi + cj \\
ij &  = & cy \\
ji &  = & (cy-nz) + mi + nj \\
j^2 &  = & -by + yi + zj \\
\end{array} 
\end{equation}with $ b, c,  m, n, y, z \in R$ satisfying the (incomplete) list of relations $cm = ny=0$, $bm = mn$, $bn = n^2$, and $mz = m^2$. 
Applying these relations to the identity $(ji)j = j(ij)$ we see:
\begin{align*}
-bny + nyi + (nz-bm)j = 0,
\end{align*}
and so $b, m, n, z$ must also satisfy $nz = bm$. We show two more relations that must be satisfied by the coefficients in the multiplication table.

\begin{lem}\label{lem:my} We have $my=0$.\end{lem}

\begin{proof} Suppose $my\not=0$, so $m, y\not=0$. Since $R$ is a domain, the relations $cm=ny=0$ imply $c=n=0$. Then $bm=0$, so $b=0$  and $m = z$. So the system reduces to
\begin{align}\label{reduced}
\begin{array}{ccl}
i^2 &  = & 0  \\
ij &  = & 0 \\
ji &  = & mi \\
j^2 &  = & yi + mj \\
\end{array}
\end{align}
Next we apply the reduced multiplication table \ref{reduced} to the product $j^3$:
\begin{align*}
j^3&= j(yi+mj) = yji + mj^2 = myi + mj^2
\end{align*}
Subtracting $mj^2$ from each side we get
\begin{align*}
myi &= j^3 - mj^2\\  &= j^2(j-m)\\ &= (yi+mj)(j-m)\\
					&= -myi + myi + m^2j - m^2j = 0\\
\end{align*}
Since $my\not=0$, this is a contradiction.\end{proof}

\begin{lem}\label{lem:cn} We have $cn=0$.  \end{lem}

\begin{proof} The relations $cm=ny=0$ require $cn=0$ except in the case that $m=y=0$. In that case, the relations $nz = 0$ and $n^2 = bn$ require that $n = 0$ or $b = n$. In the case $n$ is not zero the system reduces to
$$
\begin{array}{ccl}
i^2 &  = & ni + cj \\
ij &  = & 0 \\
ji &  = &  nj \\
j^2 &  = & 0 \\
\end{array}
$$ A computation on the identity $i^3 = (ni + cj)i$ similar to that in Lemma 5.3 shows that $cnj = 0$, from which we conclude that $cn = 0$. \end{proof}

Therefore, $A$ has multiplication defined by
\begin{equation}\label{eqn:mtable}
\begin{array}{ccl}
i^2 &  = & -cz + bi + cj \\
ij &  = & cy \\
ji &  = & (cy-bm) + mi + nj \\
j^2 &  = & -by + yi + zj \\
\end{array} 
\end{equation}\label{eqn:relations}
where $b, c, m, n, y, z$ satisfy
\begin{equation}
\begin{bmatrix}m\\ n\end{bmatrix}
{\begin{bmatrix} c& y& b-n& m-z  \end{bmatrix}}= 
\begin{bmatrix} cm& my& m(b-n)& m(m-z)\\
cn& ny& n(b-n)& n(m-z) \end{bmatrix}= 0;
\end{equation}
that is $cm=cn=ny=my=0$, $bm=mn=nz$, $n^2 = bn$ and $m^2 = mz$.

\begin{thm}\label{thm:associative}
The multiplication table {\rm \ref{eqn:mtable}} with relations \ref{eqn:relations}$.10$ determines an associative algebra. The quotient of the free $R$-algebra on $1$, $i$, and $j$ and the ideal generated by the relations in {\rm \ref{eqn:mtable}} and {\rm \ref{eqn:relations}}$.10$ is an $R$-algebra of rank $3.$
\end{thm}

\begin{proof}
Let $M$ be the rank 3 $R$-module with basis $1, i, j$ written $R\cdot 1 + Ri + Rj$. Define multiplication $M\times M\to M$ as in the table. Let $x_1$, $x_2$, and $x_3$ be elements of $M$. By comparing $(x_1x_2)x_3$ and $x_1(x_2x_3),$ we verify (in appendix A) that multiplication is associative, thus proving the first assertion.

Let $F = R\langle i, j \rangle$, be the free associative $R$-algebra in $i, j$. 
Let $A$ be the quotient of $F$ by the relations in the multiplication table \ref{eqn:mtable}  with coefficients $b, c, m, n, y, z$ that satisfy \ref{eqn:relations}$.10$:
$$
A = \frac{R\langle i, j \rangle}{\left(\begin{array}{l} i^2 + cz - bi - cj,\\ ij-cy,\\ ji-(cy-bm)-mi-nj,\\ j^2+by - yi - zj \end{array} \right)}.
$$
Let $T$ stand for the ideal generated by the relations in the table, so that $A = F/T$.
The relations in the table show that every element of $A$ may be expressed as an $R$-linear combination of $1, i$, and $j$, so $A$ is generated by $1, i, j$ as an $R$-module, and the rank of $A$ is less than or equal to 3. 

Next we construct a cubic algebra $B$. Consider the following three elements of $M_3(R)$:
$$
1 = 
\begin{bmatrix}
1& 0 & 0 \\
0& 1& 0\\
0& 0& 1
\end{bmatrix}
, \ I = 
\begin{bmatrix}
0& -cz & cy \\
1& b& 0\\
0& c& 0
\end{bmatrix}
 ,  \text{ and } J =  
\begin{bmatrix}
0& cy-bm &-by\\
0& m& y\\
1 & n& z
\end{bmatrix}.
$$
To obtain these matrices assume (for a moment) the independence of $1, i, j$. In that case we may identify $p + qi + rj\in A$ with the column vector $(p, q, t)^T\in R^3$. Then the columns of $I$ and $J$ record the action of left-multiplication by $i$ and $j$ respectively on $1, i, j$, as given in the multiplication table. However, the existence of these matrices does not depend on properties of $i, j$. Now recall that we have yet to show that that $1, i,$ and $j$ are $R$-independent.

Let $B$ be the sub $R$-algebra of $M_3(R)$ generated by $1, I, J$. By comparing the first columns of $1, I$, and $J$ we see that $1, I$, and $J$ are $R$-independent, so $B$ is free of rank 3 or greater as an $R$-algebra.

Consider the surjective $R$-linear map from $F\to B$ given by $1\mapsto 1$, $i\mapsto I$, and $j\mapsto J$. By the universal property of free algebras, this map extends to a homomorphism $\varphi:F\to B$.

We will show that $\varphi$ factors as follows:
\begin{center}
\begin{tikzpicture}
\node at (-2, 2) {$F$};
\node at (2, 2) {$B$};
\node at (0, .25) {$A$};
\node at (0, 2.5) {$\varphi$};
\node at (1.25, .75) {$\iota$};
\node at (-1.25, .75) {$\pi$};
\draw [->] (-1.5, 2) -- (1.5, 2);
\draw [->] (-1.5, 1.5) -- (-.5, .5);
\draw [->] (.5, .5) -- (1.5, 1.5);
\end{tikzpicture}
\end{center}
where $\pi:F\to A$ is the quotient map sending $x$ to the coset $x+T$, and $\iota : A\to B$ is the map $x+T\mapsto \varphi(x)$. 

It is necessary and sufficient to show that the composition $\iota\circ\pi$ is independent of representative of the coset $x + T$; that is if $x + T = x^\prime + T$, then $\varphi(x) = \varphi(x^\prime)$, which is true if and only if $T\subset \ker \varphi$. To check that $T\subset \ker \varphi$, we verify the following identities by direct computation:
\begin{align*}
I^2 &=  -cz + bI + cJ \\
IJ &=  cy \\
JI &=  (cy-bm) + mI + nJ \\
J^2 &=  -by + yI + zJ
\end{align*}
First we compute the product $I^2$:
$$I^2 = 
\begin{bmatrix}
0& -cz & cy \\
1& b& 0\\
0& c& 0
\end{bmatrix}^2 = 
\begin{bmatrix}
-cz& -cbz+c^2y & 0 \\
b& -cz+b^2& cy\\
c& bc& 0
\end{bmatrix}
$$

We compare $I^2$ to the corresponding expression in the table:
\begin{align*}
-cz + bI + cJ&= -cz \begin{bmatrix} 1& 0 & 0\\ 0& 1& 0\\ 0& 0& 1\end{bmatrix} + b \begin{bmatrix}
0& -cz & cy \\
1& b& 0\\
0& c& 0
\end{bmatrix} + c\begin{bmatrix}
0& cy-bm &-by\\
0& m& y\\
1 & n& z
\end{bmatrix}\\
&= \begin{bmatrix}
-cz& -bcz + c^2y - bcm& bcy-bcy\\
b& -cz+b^2 + cm & cy\\
c & bc + cn & -cz+cz
\end{bmatrix}
\end{align*}
Applying the relations $cm=cn=0$ and simplifying we observe that indeed
$ I^2 = -cz + bI + cJ .$ 

Next we verify that $IJ = cy$. We compute the product
\begin{align*}
IJ &=
 \begin{bmatrix}
0& -cz & cy \\
1& b& 0\\
0& c& 0
\end{bmatrix}
\begin{bmatrix}
0& cy-bm &-by\\
0& m& y\\
1 & n& z
\end{bmatrix}
	= \begin{bmatrix}
cy & -cmz & cny\\
0 & cy - bm+bm & -by + by\\
0& cm & cy
\end{bmatrix}.
\end{align*}
Simplifying and applying the relations $cm=ny=0$, we obtain
$$
\begin{bmatrix}
cy& 0 & 0 \\
0 & cy& 0 \\
0 & 0 & cy
\end{bmatrix}
 = 
cy 
\begin{bmatrix}
1& 0 & 0 \\
0& 1& 0\\
0& 0& 1
\end{bmatrix}
$$
Thus $IJ = cy$.

Similar computations demonstrate that $\varphi$ respects the relations $ji = cy-bm+mi+nj$ and $j^2 = -by+yi+zj$. Therefore, $T\subset \ker \varphi$. 

Since $\varphi$ is surjective and $\varphi$ factors as $\iota\circ\pi$, we have that $\iota:A\to B$ is a surjective algebra homomorphism $A\to B$. Suppose $\iota(p+qi+rj+ T) = \varphi(p+qi+rj) = 0$. Then $p + qI + rJ = 0$. By the independence of $1, I, J$, we have $p = q = r = 0$. Thus $p+qi+rj+ T = 0+T$, which shows that $\iota$ is injective, and hence bijective. Then $\iota$ is an algebra isomorphism $A\to B$, and $A$ is a cubic algebra.
\end{proof}

\begin{lem}\label{lem:cases} The relations {\rm ({\ref{eqn:relations}.10)}} on the coefficients of the multiplication table of $A$ {\rm(\ref{eqn:mtable})} require, in particular, one of the two following cases:
\begin{description}
\item[(C)] $m = n = 0$, in which $A$ is commutative, or 
\item[(E)]  $m, n$ are not both $0$, and $c = y = 0$, in which $A$ is noncommutative.
\end{description}
\end{lem}
\begin{proof}To satisfy $cn=0$, first let $c = 0$. Then by $my=ny=0$, we have cases $c = m = n = 0$ or $c = y = 0$. Next satisfy $cn=0$ by taking $n = 0$. So by $cm = my = 0$ we have either $m = n = 0$ or $c = n = y = 0$.
\end{proof}

In case (C) we have the following multiplication table for $A$
\begin{equation}\label{eqn:commutative}
\begin{array}{ccl}
i^2 & = & -cz+bi + cj\\
ij & = & cy\\
ji& = & cy\\
j^2&= &-by + yi + zj
\end{array}
\end{equation}
where $b, c, y, z$ are arbitrary elements of $R$. Observe that 
  $ij = ji$, and so $A$ is commutative. These algebras are examined in detail in Gross and Lucianovic \cite{Grossand}. For completeness we include this result from their paper.

\begin{definition}\label{def:binaryform}
A \emph{binary cubic form} is an expression $p(x, y)\in R[x, y]$ such that $p(x, y)$ is homogeneous of degree $3$ in $2$ variables:
$$
p(x, y) = ax^3 + bx^2y + cxy^2 + dy^3.
$$
The \emph{discriminant} $\Delta(p)$ of $p$ is given by
$$
\Delta(p) = b^2c^2 + 18abcd - 4ac^3-4db^3-27a^2d^2 \in R/R^{\times}
$$
\end{definition}
The set $M$ of binary cubic forms over $R$ is a free $R$-module of rank 4, denoted Sym$^3(R^2)$. We define the (twisted) action of $g =  \begin{bmatrix}\alpha&\beta\\ \gamma&\delta \end{bmatrix}\in \rm{GL}_2(R)$ on $M$ by
$$
p\mapsto g\cdot p = \frac{1}{\det g}p(\alpha i + \gamma j, \beta i + \delta j)
$$
This action is faithful, and $\Delta(g\cdot p) = (\det g)^2\Delta(p)$.

\begin{prop}[Gross and Lucianovic]
There is a discriminant-preserving bijection between the set of orbits of the twisted action of GL($R$) on the $R$-module {\rm Sym}$^3(R^2)$ and the set of isomorphism classes of commutative cubic rings $A$ over $R$.
\end{prop}
\begin{rem}\label{rem:grossletters}
In their paper {\rm \cite{Grossand}} Gross and Lucianovic use different characters in their multiplication table for commutative cubic algebras. The notation correspondence from this paper to theirs is $z\mapsto -c$, $x\mapsto -a$, $b\mapsto b$, and $y\mapsto d$.
\end{rem}

Next consider case (E), in which $c = y = 0$ and $m, n$ are not both zero. If $m\not=0$, from $bm=mn$ we obtain $b=n$, and from $m^2 = mz$ we see that $m=z$. Taking $n\not=0$ produces identical results. Thus in case (E) the multiplication table for $A$ is
\begin{equation}\label{eqn:exceptional}
\begin{array}{ccl}
i^2 & = & ni\\
ij&=&0\\
ji&=&-mn + mi + nj\\
j^2&=&mj
\end{array}
\end{equation}
where $m, n$ are arbitrary elements of $R$ (with $m, n$ not both 0). Comparing $ij$ and $ji$ we see that $A$ is not commutative in this case.

 \begin{prop}\label{prop:exceptional} In case {\rm (E)}, algebra $A$ is an exceptional ring.\end{prop}
 \begin{proof}
Let $I = n-i$. Note that $iI = i^2 - ni = 0$ and $jI = nj-ji =  mI$. Let $M$ be the left ideal $(I, j)\subset A$. We check that $A = R\oplus M$:
\begin{align*}
M = AI + Aj &= (R + Ri + Rj)(n-i) + (R + Ri + Rj)j\\
	&= R(n-i) + Rm(n-i) + Rj + Rmj\\
	&= RI + Rj.
\end{align*} 
Then $A = R + RI + Rj.$ By independence of $1, I, J$ we have $A=R\oplus M$.

Consider left-multiplication in $M$:
\begin{align*}
I^2 &= n^2 - ni - ni + i^2 = n^2 - ni = nI\\
Ij  &= nj-ij = nj\\
jI&= nj - ji = nj + mn - mi - nj = mI\\
j^2&= mj 
\end{align*}
 Thus left-multiplication in $M$ factors through the $R$-linear map $t(I) = n$, $t(j) = m$, which shows that $A$ is an exceptional ring.
\end{proof}

It follows from Proposition \ref{prop:exceptionalinvolution} that in case  {\rm(E)}, $A$ possesses a standard involution.

\begin{prop}\label{prop:exceptionalstandard}
In case {\rm(E)}, the algebra $A$ possesses the following standard involution:
$
\overline{1}= 1, \
\overline{i}= n-i, \
\overline{j}= m-j,
$
extended linearly.
\end{prop}
\begin{proof}
First we check that $^-$ is an involution. Observe that $^-$ fixes $R$.  Since $\overline{\overline{i}} = \overline{n-i} = n - (n-i) = i$ (similarly, $\overline{\overline{j}} = j$), $^-$  is self-inverse.  Since $^-$ is $R$-linear, to check that $\overline{xy} = \overline{y}\, \overline{x}$ for all $x, y\in A$, it suffices to check that $\overline{i^2} = \overline{i}^2$, $\overline{j^2} = \overline{j}^2$,  $\overline{ij} = \overline{j}\, \overline{i}$ and $\overline{ji} = \overline{i}\, \overline{j}$. First
\begin{align*}
\overline{i^2} = \overline{ni} = n^2 - ni = n^2 - ni - ni + ni = (n-i)^2 = \overline{i}^2.
\end{align*}
By symmetry, $\overline{j^2} = \overline{j}^2$.

Next we check
\begin{align*}
\overline{ij} &= \overline{0} = 0 = mn - mi - nj + ji = (m-j)(n-i) = \overline{j}\, \overline{i},
\end{align*}
and also
\begin{align*}
\overline{ji} &= \overline{-mn + mi + nj} = -mn + m(n-i) + n(m-j) = mn -mi  - nj\\
		&= (n-i)(m-j) = \overline{i}\, \overline{j}.
\end{align*}
Thus $^-$ is an involution. To check that $^-$ is standard, take $a + bi + cj\in A$. Then
\begin{align*}
(a+bi + cj)(\overline{a+bi + cj})&= (a+bi + cj)(a + b(n-i) + c(m-j))\\
	&=  a(a + bn + cm) - abi - acj \\
	&\ \ \ + abi + b^2ni + bcmi - b^2i^2 - bcij\\
	&\ \ \ + acj + bcnj + c^2mj - bcji - c^2j^2\\
	&= a(a + bn + cm) +bcmi + bcnj  - bcji\\
	&= a(a + bn + cm) +bcmi + bcnj  - bc(-mn + mi + nj)\\
	&= a(a + bn + cm)  + bcmn\in R \\
\end{align*}\end{proof}

In Proposition \ref{prop:exceptionalstandard} we have shown that a noncommutative cubic algebra possesses a standard involution.
\begin{lem}\label{lem:nilproduct}
Suppose $A$ possesses a standard involution and is commutative. Then $A$ has multiplication table $i^2= ij = ji = j^2 = 0$ ($A$ is the nilproduct ring).
\end{lem}
\begin{proof}
Since $A$ is commutative, we have $m= n=0$.  Then $A$ has multiplication table:
\begin{equation*}
\begin{array}{ccl}
i^2 &  = & -cz + bi + cj \\
ij &  = & cy \\
ji &  = & cy \\
j^2 &  = & -by + yi + zj \\
\end{array} 
\end{equation*}
Since $A$ has a standard involution, $i$ and $j$ satisfy monic quadratic equations in $R[x]$; i.e., there exist $p, q, r, s\in R$ such that $i^2 = pi-q$ and $j^2 = rj-s$. Equating coefficients with those in the above multiplication table we see that $p = b$, $c=0$, $r = z$, and $y = 0$, so that $i^2 = bi$ and $j^2 = zj$. Since $i+\overline{i}, j+\overline{j}\in R$, there exist $t, u\in R$ such that $\overline{i} = t-i$ and $\overline{j} = u-j$. Then since $i\overline{i}, j\overline{j}\in R$ we see that $\overline{i} = b-i$ and $\overline{j} = z-j$.

Next recall that $\overline{ij} = \overline{j}\, \overline{i}$, so that 
\begin{align*}
\overline{cy}&= cy = (z-j)(b-i) = bz - zi-bj + ji = bz - zi - bj + cy
\end{align*}
Subtracting $cy$ from each side we obtain
$
0 = bz - zi - bj
$. Then $b = z = 0$, which shows that $A$ is the nilproduct ring.\end{proof}
Thus we have shown that a cubic algebra $A$ is either commutative or possesses a standard involution, and that the unique commutative algebra with a standard involution is the nilproduct ring.

Thus we have the following picture of the parameter space of cubic algebras; the commutative algebras (C) have four degrees of freedom (multiplication table \ref{eqn:commutative}),  the algebras which are exceptional rings (E) have two degrees of freedom (multiplication table \ref{eqn:exceptional}), and the intersection of the two is the unique nilproduct ring $\mathcal{N}$.

\hspace{-12 mm}
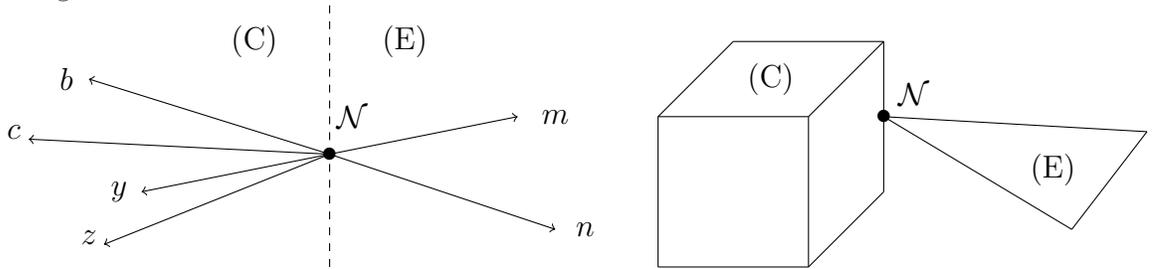
\begin{tabular}{lr}
\begin{tikzpicture}
\node at (-1, 2.5) {(C)};
\node at (1, 2.5) {(E)};
\node at (0, 1) {$\bullet$};
\node at (.3, 1.5) {$\mathcal{N}$};
\node at (3, 1.5) {$m$};
\node at (3.4, 0) {$n$};
\node at (-4.2, 1.3) {$c$};
\node at (-3.5, 2) {$b$};
\node at (-2.8, .5) {$y$};
\node at (-3.2, -.1) {$z$};
\draw [->] (0, 1) -- (3, 0);
\draw [->] (0, 1) -- (2.5, 1.5);
\draw [->] (0, 1) -- (-3, -.2);
\draw [->] (0, 1) -- (-2.5, .5);
\draw [->] (0, 1) -- (-3.2, 2);
\draw [->] (0, 1) -- (-4, 1.2);
\draw [-] [dashed] (0,-.5) -- (0,3);
\end{tikzpicture}
\ \ &
\begin{tikzpicture}
\node at (0, 1) {$\bullet$};
\node at (-1.5, 1.5) {(C)};
\node at (2.25, .3) {(E)};
\node at (.4, 1.3) {$\mathcal{N}$};
\draw [-] (0, 0) -- (0, 2);
\draw [-] (-2,2) -- (0,2);
\draw [-] (0, 2) -- (-1, 1);
\draw [-] (-3, 1) -- (-2, 2);
\draw [-] (-1, 1) -- (-3, 1);
\draw [-] (0,0) -- (-1,-1);
\draw [-] (-1,-1) -- (-3,-1);
\draw [-] (-1,1) -- (-1,-1);
\draw [-] (-3, -1) -- (-3, 1);
\draw [-] (0, 1) -- (2.5, -.5);
\draw [-] (0, 1) -- (3.5, .8);
\draw [-] (2.5, -.5) -- (3.5, .8);
\end{tikzpicture}
\end{tabular}

\begin{thm}[Voight]\label{thm:cubicexceptional} A cubic algebra is an exceptional ring if and only if it possesses a standard involution.\end{thm}
\begin{proof}

In Proposition \ref{prop:exceptionalinvolution} we showed that exceptional rings possess a standard involution. 

Next suppose that $A$ possesses a standard involution. If $A$ is commutative (by Lemma \ref{lem:nilproduct}), $A$ is the nilproduct ring, which is exceptional. If $A$ is noncommutative, then $A$ is in case (E) of Lemma \ref{lem:cases}. Algebras in case (E) were shown to be exceptional rings in Proposition \ref{prop:exceptional}.
\end{proof}

\begin{ex}\label{ex:charpoly} Specializing to case {\rm(E)}, we compute the characteristic polynomial of an arbitrary element $p + qi + rj\in A$. Since $p + qi + rj$ satisfies a quadratic equation, we expect to see a repeated factor.
From multiplication table \ref{eqn:exceptional} we have the following left-regular representation of $A:$
$$
1 =
\begin{bmatrix}
1& 0 & 0 \\
0& 1& 0\\
0& 0& 1
\end{bmatrix}, \ I = 
\begin{bmatrix}
0& 0 & 0 \\
1& m& 0\\
0& 0& m
\end{bmatrix},  \text{ and } J = 
\begin{bmatrix}
0& 0 &0\\
0& n& 0\\
1 & 0& n
\end{bmatrix},
$$
and so $p + qi + rj$ corresponds to the lower-triagular matrix
$$
\begin{bmatrix}
p& 0 & 0\\
q& p+mq +nr& 0\\
r& 0& p +mq + rn
\end{bmatrix}.
$$
Then the characteristic polynomial $P(T)$ of $p + qi + rj$ is $(T-p)(T - p-mq-rn)^2$.
\end{ex}

\begin{prop}\label{cubicfield}
If $A$ is a cubic algebra over a field $F$ and $A$ possesses a standard involution, then $A$ is the nilproduct ring or isomorphic to the algebra of upper-triangular $2\times 2$ matrices over $F$.
\end{prop}

\begin{proof}
In the special case that $R$ is a field $F$ (so $A$ is a vector space), we can compute the isomorphism classes of exceptional $F$-algebras. For $(n, m)\in F\times F$, let $A$ be the exceptional ring with multiplication given by $i^2 = ni, ij = nj, ji = mi, j^2 = mj$. For $(N, M)\in F\times F$, the equation
\begin{align*}
\begin{bmatrix} \alpha & \beta\\ \gamma & \delta \end{bmatrix} \begin{bmatrix} n\\ m \end{bmatrix} =
\begin{bmatrix} N \\ M \end{bmatrix}
\end{align*}
has a solution $[\begin{smallmatrix} \alpha & \beta\\ \gamma & \delta \end{smallmatrix}]\in GL_2(F)$ if $(n, m)\not=(0,0)$ and $(N, M)\not=(0,0)$, or is true for any $\alpha, \beta, \gamma, \delta\in F$  if $(n, m)=(N, M)=(0,0)$. Therefore the two isomorphism classes are the singleton class of the nilproduct ring and all other exceptional rings over $F$.

Consider $T$, the ring of upper triangular $2\times 2$ matrices over $F$. Observe that $T$ is an $F$-algebra of rank $3$ with basis 
$1 = [\begin{smallmatrix} 1& 0\\ 0& 1  \end{smallmatrix}],$
$i = [\begin{smallmatrix} 0& 0\\ 0& 1  \end{smallmatrix}],$
 $j = [\begin{smallmatrix} 0& 1\\ 0& 0  \end{smallmatrix}].$ 
 Computing the multiplication table: 
 \begin{align*}
 i^2 &= i\\
 ij &= 0\\
 ji&= j\\
 j^2&= 0
 \end{align*}
we see that $T$ is not the nilproduct ring. As in example \ref{ex:adjoint}, the adjoint map is a standard involution on $T$. Therefore 
cubic $F$-algebras which possess a standard involution are either the nilproduct ring or isomorphic to the algebra of upper triangular $2\times 2$ matrices.
\end{proof}

\section{Semisimple algebras with standard involutions}

In this section we use Wedderburn's theorem and an extended version of Frobenius's theorem to classify semisimple algebras over a field $F$ (with char$(F)\not= 2$) that have a standard involution. We show that such an algebra must be $F$, a quadratic extension of $F$, a quaternion algebra over $F$, or $F\times F$. These results are known \cite{Voightquat}. Let $A$ be an $F$-algebra. 

\begin{definition}\label{def:simple}
An $F$-algebra is \emph{simple} if it has no nontrivial (two-sided) ideals. We say $A$ is \emph{semisimple} if $A$ is a finite direct product of simple algebras {\rm\cite{dandf}}. 
\end{definition}

\begin{definition}\label{def:degree}
The \emph{degree} of $A$ is the smallest integer $m\in \mathbb{Z}_{\ge 1}$ such that every element $x\in A$ satisfies a monic polynomial $f(T)\in F[T]$ of degree $m$; if no such integer exists, we say $A$ has degree $\infty$. Define the \emph{degree} of an element $x\in A$ to be the degree of the minimal polynomial of $x$, or $\infty$ if no such polynomial exists.
\end{definition}

\begin{ex}\label{ex:lowdegree}
If $A$ has degree 1, then $A = F$. If $A$ possesses a standard involution, then since each $x\in A$ satisfies a monic quadratic in $F[T]$, the degree of $A$ is at most 2.\end{ex}

The following proof is similar to the proof of Frobenius's theorem (Theorem \ref{thm:frobenius}) and appears in \cite{Voightquat}. We include it for completeness.

\begin{thm}[Frobenius, Extended]\label{thm:extendedfrobo}
Let $A$ be a division $F$-$algebra$ of degree $2$ and suppose that {\rm char}$(F)\not=2$. Then either $A$ is a quadratic field extension of $F$ or $A$ is a quaternion $F$-algebra.
\end{thm}
\begin{proof}
Since the degree of $A$ is not 1, $A\not=F$ by Example \ref{ex:lowdegree}. Let $i\in A\setminus F$. Since the degree of $A$ is 2 there exist $t, n\in F$ such that $i^2 = -n+ti$ $(n\not=0)$. Thus the extension $K = F(i)$ over $F$ is quadratic. Since $K = F\oplus Fi$ is a field we may consider $A$ as a (left) $K$-vector space. Then $A = K$ if and only if dim$_F(A) = 2$, if and only if dim$_K(A) = 1$.  
 
Since $2\in F^\times$, as in Lemma \ref{lem:completesquare} we may take $i^2 = a$ for some $a\in F^\times$ by ``completing the square." Let $\varphi:A\to A$ be conjugation by $i$; $\varphi(x) = ixi^{-1}$. Note that since $i^2 = a$, we have $a^{-1} = i^{-2}$ and $a^{-1}i = i^{-1}$. Thus $\varphi(x) = a^{-1}ixi$. A check of the axioms shows that $\varphi$ is a $K$-linear endomorphism of $A$. Observe that $\varphi\circ\varphi$ is the identity on $A$. Since $\varphi^2 - 1 = 0$ in End$_A$, the roots of the minimal polynomial of $\varphi$ are distinct, so $\varphi$ is diagonalizable with eigenvalues $\pm 1$. Let $A^-$ be the eigenspace of $-1$ and $A^+$ be the eigenspace of $1$, so that $A = A^+\oplus A^-$. Since $\varphi(i) = i$ we have $K\subset A^+$.

Next we show that dim$_K A^+ = 1$. 
Let $x\in A^+$. Consider the field $L = F(i, x)$. The degree of $L/F(i)$ is at most 2 because $x$ satisfies a quadratic in $F\subset F(i)$. Then the degree of $L/F$ is finite (at most 4). Let $y\in L\setminus F$ and let $f(t) = t^2 - bt + c\in F[t]$ be the minimum polynomial of $y$. Suppose $f(t)$ factors in $L$ as $(t-y)^2$. Then $2y = b\in F$, a contradiction. Therefore $L$ is separable. Since $L$ is finite and separable over $F$, by the primitive element theorem there exists $z\in L$ such that $L = F(z)$. Since $z$ satisfies a quadratic polynomial in $F[t]$, $L$ is at most a quadratic extension of $F$. Together with the fact that $L$ is an extension of $K$ we see that dim$_F(L) = 2$, and so $L = K$. Since $K\subset A^+$ and $K(x) = K$ for all $x\in A^+$ we obtain  dim$_K A^+ = 1$, so $A^+ = K$.

Now we show that dim$_K A^- = 1$. Let $j\in A^-\setminus\{ 0\}$. Then $\varphi(j) = -j = iji^{-1}$, so $jij^{-1} = -i$, and hence all elements of $A^-$ conjugate $i$ to $-i$. Also note that $j^{-1}\in A^-$. Let $p, q\in A^-$. Then 
$$
(pq)i(pq)^{-1} = p(q i  q^{-1})p^{-1} = p(-i)p^{-1} = i,
$$
so the product of two elements of $A^-$ is in $A^+$ and also is in the centralizer of $i$. Let $k\in A^-$. Since $k = (kj^{-1})j$ we see that any element of $A^-$ differs multiplicatively from $j$ by an element of $A^+ = K$. Thus $A^- \cong Kj$ and so dim$_K A^- = 1$.

Finally we show that $A$ is a quaternion algebra. By completing the square, we may assume that $j^2 = b\in F^\times$. Since $j\in A^-$ we have $iji^{-1} = -j$ so $ij = -ji$. Therefore $A\cong \left(\dfrac{a, b}{F}  \right)$.
\end{proof}

\begin{cor}\label{cor:divisionstandard}
If $D$ is a degree $2$ division $F$-algebra then $D$ possesses a standard involution.
\end{cor}
\begin{proof} By Theorem \ref{thm:extendedfrobo}, $D$ is either $F$, a quadratic field extension of $F$, or a quaternion algebra over $F$.
If $D = F$ then the identity is a standard involution. If $D$ is a quadratic field extension over $F$ then $D$ is a quadratic algebra over $F$. We showed that quadratic algebras possess a (unique) standard involution in section 2.3. If $D$ is a quaternion algebra, then $D$ has a standard involution by Example \ref{ex:quaternioninvolution}.\end{proof}

\begin{lem}\label{lem:involutionMnD}
Let $D$ be a division algebra containing field $F$ and consider the matrix ring $M_n(D)$ as an $F$-algebra. If $M_n(D)$ possesses a standard involution then $n\le 2$, and if $n=2$ then $D = F$.
\end{lem}
\begin{proof}
Let $n=2$. By Example \ref{ex:adjoint}, $M_2(F)$ has a standard involution (the adjoint map). Let $i\in D\setminus F$ and consider the matrix $i\cdot 1 = [i]$. Since $M_2(D)$ possesses a standard involution, $[i]$ has degree at most 2. So there exist $t, n\in F$ such that $[i]^2 - t[i] + n = 0$. Then in particular, $i^2- ti + n = 0$. Hence every element $i\in D\setminus F$ satisfies a monic quadratic polynomial in $F[x]$. Since $2\in F^\times$ we can complete the square and assume $i^2 = a\in F^\times$. 

Now consider $[\begin{smallmatrix} 1&0\\ 0& i \end{smallmatrix}]\in M_2(D)$. There exist $T, N\in F$ such that $[\begin{smallmatrix} 1&0\\ 0& i \end{smallmatrix}]$ satisfies $x^2 - Tx + N\in F[x]$. Then in particular we have
\begin{align*}
1 - T+N = 0 = i^2 - Ti + N,
\end{align*}
from which we see that $N = -a$ and $T = 0$, and so $a = 1$. From $i^2 = 1$, it follows that $(i-1)(i+1) = 0$. Since $D$ is a division algebra $i$ must equal $\pm 1$, a contradiction. Therefore if $n=2$ then $D = F$.

Suppose $n\ge 3$. Let $M\in M_n(D)$ be the diagonal matrix with diagonal entries $a_{1,1}, a_{2,2}, \ldots, a_{n,n}= 0, -1, 1, 1,  \ldots, 1$. The characteristic polynomial $c_{M}(x)$ of $M$ is $x(x+1)(x-1)^{n-2}$. Since the characteristic polynomial has three distinct roots ($0, \pm 1$), the minimal polynomial $\mu_M(x)$ of $M$ has deg$(\mu_M)\ge 3$, a contradiction\\ \cite[Section 12.2]{dandf}.
\end{proof}

Wedderburn's theorem greatly simplifies our classification.
\begin{thm}[Wedderburn]\label{thm:Wedderburn}
Let $B$ be a simple $F$-algebra. Then $B$ is isomorphic to $M_n(D)$ for some unique (up to isomorphism) division $F$-algebra $D$ and unique $n\in \mathbb{Z}_{\ge 1}$.
\end{thm}

\begin{thm}\label{thm:simple} Let $A$ be a simple $F$-algebra of degree less than or equal to $2$, with {\rm char}($F)\not=2$.  Then $A$ is isomorphic to $F$, a quadratic field extension of $F$, or a quaternion algebra over $F$.
\end{thm}
\begin{proof}If deg$(A) =1$ then $A = F$. Suppose deg$(A) =2$.
By Wedderburn's theorem, $A\cong M_n(D)$ for a unique (up to isomorphism) division $F$-algebra $D$ and a unique $n\in \mathbb{Z}_{\ge 1}$. If $n=1$ then $A$ is a division $F$-algebra, of degree 2. Then by Theorem \ref{thm:extendedfrobo}, $A$ is is isomorphic to $F$, a quadratic field extension of $F$, or a quaternion algebra over $F$. For $n\not=1$ we apply Lemma \ref{lem:involutionMnD} to complete the proof.
\end{proof}Note that $M_2(F)$ is a quaternion algebra \cite{Voightquat}.

We include two remarks about degree.

\begin{rem}\label{prop:relativedegree} Let $A$ and $B$ be $F$-algebras of finite degree. Then {\rm deg}$(A\times B) = {\rm deg}(A)+{\rm deg}(B)$ if and only if there exist $x\in A$ and $y\in B$ such that  {\rm deg}$(x) = {\rm deg}(A)$, {\rm deg}$(y) = {\rm deg}(B)$, and the respective minimal polynomials of $x$ and $y$ are relatively prime.
\end{rem}
\begin{proof} Let $\mu_x$ denote the minimal polynomial of $x$ in $F[T]$.
Let $x\in A$ and $b\in B$ satisfy  {\rm deg}$(x) = {\rm deg}(A)$, {\rm deg}$(y) = {\rm deg}(B)$, and that $\mu_x$ and $\mu_y$ are relatively prime. Note that $\mu_x\mu_y$ has degree deg$(x) + {\rm deg}(y)$.  
If $(x, y)$ satisfies a polynomial $p(T)\in F[T]$, then $\mu_x, \mu_y\mid p(T)$. Then since $\mu_1, \mu_2$ are relatively prime, deg$(p)\ge {\rm deg}(x) + {\rm deg}(y)$, so deg$((x, y))\ge {\rm deg}(x) + {\rm deg}(y)$.  Since $(x, y)$ satisfies $\mu_x\mu_y$ in particular, we have deg$((x, y)) = {\rm deg}(x) + {\rm deg}(y)$. Thus deg$(A\times B)\ge {\rm deg}(A) + {\rm deg}(B)$. Let $(a, b)\in A\times B$. Since $(a, b)$ satisfies $\mu_a\mu_b$ and deg$(\mu_a\mu_b)\le {\rm deg}(A) + {\rm deg}(B)$, we have deg$(A\times B)\le {\rm deg}(A) + {\rm deg}(B)$.

Note that for all $a\in A$ and $b\in B$ we have deg$(a)\le {\rm deg}(A)$ and deg$(b)\le {\rm deg}(B)$, so that deg$(\mu_a\mu_b)\le {\rm deg}(A) + {\rm deg}(B)$, with equality if and only if deg$(a) = {\rm deg}(A)$ and deg$(b) = {\rm deg}(B)$. Suppose that {\rm deg}$(A\times B) = {\rm deg}(A)+{\rm deg}(B).$ Take $(a, b)\in A\times B$ such that the minimal polynomial $\mu_{a,b}$ of $(a, b)$ has degree ${\rm deg}(A)+{\rm deg}(B)$. Since $(a, b)$ satisfies $\mu_a\mu_b$ we have deg$(\mu_a\mu_b)\ge {\rm deg}(A)+{\rm deg}(B)$.
Therefore deg$(a) = {\rm deg}(A)$ and deg$(b) = {\rm deg}(B)$, and since deg$(\mu_a\mu_b) = {\rm deg}(a) + {\rm deg}(b)$ we see that $\mu_a$ and $\mu_b$ are relatively prime.
 \end{proof}
 \begin{rem}
For two $F$-algebras of finite degree $A$ and $B$ it is not always true that  {\rm deg}$(A\times B) = {\rm deg}(A)+{\rm deg}(B)$. For example, for $k\ge2$, consider the boolean ring $\mathbb{F}_2^k$. For all $x\in \mathbb{F}_2^k$ we have $x^2 + x = 0$, so deg$(\mathbb{F}_2^k) = 2$, not $2k$.
 In Proposition \ref{prop:relativedegree} the hypothesis that the respective minimal polynomials of $x$ and $y$ are relatively prime avoids such exceptions.
 \end{rem}

\begin{prop}\label{prop:semisimple} Let $A$ be a nonsimple semisimple $F$-algebra of degree $2$, with {\rm char}($F)\not=2$.  Then $A$ is isomorphic to $F\times F$.
\end{prop}
\begin{proof} 
An example of such an algebra is $F\times F$ which has standard involution $(x,y)\mapsto (y, x) = \overline{(x, y)}$ (Example \ref{ex:FcrossF}). So $F\times F$ has degree 2.

By Wedderburn's theorem, $A\cong M_{n_1}(D_1)\times M_{n_2}(D_2)\times \dots \times M_{n_k}(D_k)$ where $D_1, D_2, \ldots, D_k$ are division $F$-algebras  uniquely determined (up to isomorphism) by $F$, as are the integers $n_1, n_2, \ldots, n_k$.

Let $x = (a_1, a_2, \ldots, a_k)\in A$. Since deg$(A) = 2$, there exist $t, n\in F$ such that $x^2 - tx + n = 0$. In particular, $a_i^2 - ta_i + n = 0$ for $1\le i\le k$. Since $x$ is arbitrary it follows that deg$\left( M_{n_i}(D_i) \right)\le 2$ for $1\le i\le k$. Then by Theorem \ref{thm:simple}, $M_{n_i}(D_i)$ is isomorphic to $F$, or a quadratic field extension of $F$, or a quaternion algebra over $F$.  

Next we show that $M_{n_j}(D_j) = F$ for $1\le j\le k$.
Let $i\in D_j\setminus F$. Since $2\in F^\times$ we can (up to isomorphism) take $i^2 = a\in F^\times$ by completing the square. Consider the matrix $i\cdot 1\in M_{n_j}(D_j)$.  Let $v = (1, 1, \ldots, 1, i\cdot 1, 1, \ldots,  1)\in A$.  Let $t, n\in F$ satisfy $v^2 - tv + n = 0$. Then (by considering the matrix entries)  since $k>1$ we have $1 - t + n = 0$ and $a - ti + n = 0$ in $D_j$, and so $t = 0$ and $n= -1$. Therefore $a-1 = i^2 - 1 = 0$, from which we see that $(i-1)(i+1) = 0$. Since $D_j$ is a division algebra, $i = \pm 1$, a contradiction to $i\in D_j\setminus F$. Therefore $D_j = F$.

Finally we show that $n\le 2$.  Let $v = (0,-1, 1, 1,  \ldots, 1)\in A$. Let $t, n\in F$ satisfy $v^2 - tv + n = 0$. From the first coordinate we obtain $n = 0$. From the second we get $1+t = 0$, and from the third coordinate we see that $1-t = 0$, a contradiction. Therefore $n\le 2$.
\end{proof}

\section{Future work}

We would like to extend our investigation to algebras of high (greater or equal to 4) rank. Much work has been done in in rank 4 because of the importance of quadratic algebras. However, little is known about algebras of rank 5 or greater. In rank 5 we may be able to write down a multiplication table, (as we did in rank 3), which we could then use to find isomorphism classes. It may be that the results for ranks 5 and lower suggest a general result for finite-rank algebras.

\appendix

\section{Appendix}
Here we include the computations showing associativity for Theorem \ref{thm:associative}.
By Lemma $\ref{lem:cases}$, we may consider case (C), in which $m=n=0$, and case (E), in which $c=y=0$.
In case (C) we have the following multiplication table for $A$
\begin{equation*}\label{eqn:commutative}
\begin{array}{ccl}
i^2 & = & -cz+bi + cj\\
ij & = & cy\\
ji& = & cy\\
j^2&= &-by + yi + zj
\end{array}
\end{equation*}
where $b, c, y, z$ are arbitrary elements of $R$.

Let $x_1 = p_1 + q_1i + r_1j$, $x_2 = p_2 + q_2i + r_2j$, and $x_3 = p_3 + q_3i + r_3j$ be three elements of $M$.
First we compute $(x_1x_2)x_3$:
\begin{align*}
[(p_1 + q_1i + r_1j)(p_2 + q_2i + r_2j)](p_3 + q_3i + r_3j) \ \ \ \ \ \ \ \ \ \ \ \ \ \ \ \ \ \ \ \ \ \ \ \ \ \ \ \ \ \ \ \ \ \ \ 
\end{align*}
\vspace{-14 mm}
\begin{align*}
&= \left( \begin{array}{l}p_1p_2 + p_1q_2i +p_1r_2j \\  
					+ q_1p_2i + q_1q_2i^2 + q_1r_2ij\\
					+ r_1p_2j + r_1q_2ji + r_1r_2j^2\\
\end{array} \right) (p_3 + q_3i + r_3j)\\
&=  \left( \begin{array}{l}p_1p_2 + (p_1q_2 + q_1p_2)i + (p_1r_2 + r_1p_2)j \\
		+ q_1q_2(-cz + bi + cj)\\
		+ q_1r_2(cy)\\
		+ r_1q_2(cy)\\
		+ r_1r_2(-by + yi + zj )
		\end{array}\right)(p_3 + q_3i + r_3j)\\
\end{align*}
\begin{align*}
&= \left(\begin{array}{l}p_1p_2 - q_1q_2cz+ q_1r_2cy + r_1q_2cy - r_1r_2by\\
	+(p_1q_2 + q_1p_2+q_1q_2b+r_1r_2y)i\\
	+(p_1r_2 + r_1p_2+ q_1q_2c+r_1r_2z)j\end{array}\right)(p_3 + q_3i + r_3j)\\
&= \left(\begin{array}{l}
	p_1p_2p_3 - q_1q_2p_3cz+ q_1r_2p_3cy + r_1q_2p_3cy - r_1r_2p_3by\\
	+(p_1q_2p_3 + q_1p_2p_3+q_1q_2p_3b+r_1r_2p_3y)i\\
	+(p_1r_2p_3 + r_1p_2p_3+ q_1q_2p_3c+r_1r_2p_3z)j\end{array}\right)\\
	&\hspace{4 mm}+ \left(\begin{array}{l}
	(p_1p_2q_3 - q_1q_2q_3cz+ q_1r_2q_3cy + r_1q_2q_3cy - r_1r_2q_3by)i\\
	+(p_1q_2q_3 + q_1p_2q_3+q_1q_2q_3b+r_1r_2q_3y)i^2\\
	+(p_1r_2q_3 + r_1p_2q_3+ q_1q_2q_3c+r_1r_2q_3z)ji\end{array}\right)\\
	&\hspace{4 mm}+\left(\begin{array}{l}
	(p_1p_2r_3 - q_1q_2r_3cz+ q_1r_2r_3cy + r_1q_2r_3cy - r_1r_2r_3by)j\\
	+(p_1q_2r_3 + q_1p_2r_3+q_1q_2r_3b+r_1r_2r_3y)ij\\
	+(p_1r_2r_3 + r_1p_2r_3+ q_1q_2r_3c+r_1r_2r_3z)j^2\end{array}\right)\\
&= p_1p_2p_3 - q_1q_2p_3cz+ q_1r_2p_3cy + r_1q_2p_3cy - r_1r_2p_3by\\
	&\hspace{4 mm}+\left(\begin{array}{l}
	p_1q_2p_3 + q_1p_2p_3+q_1q_2p_3b+r_1r_2p_3y\\
	+p_1p_2q_3 - q_1q_2q_3cz+ q_1r_2q_3cy + r_1q_2q_3cy - r_1r_2q_3by
	\end{array}\right)i\\
	&\hspace{4 mm}+\left(\begin{array}{l}
	p_1r_2p_3 + r_1p_2p_3+ q_1q_2p_3c+r_1r_2p_3z\\
	+p_1p_2r_3 - q_1q_2r_3cz+ q_1r_2r_3cy + r_1q_2r_3cy - r_1r_2r_3by
	\end{array}\right)j\\
	&\hspace{4 mm}+(p_1q_2q_3 + q_1p_2q_3+q_1q_2q_3b+r_1r_2q_3y)(-cz+bi+cj)\\
	&\hspace{4 mm}+(p_1r_2q_3 + r_1p_2q_3+ q_1q_2q_3c+r_1r_2q_3z)(cy)\\
	&\hspace{4 mm}+(p_1q_2r_3 + q_1p_2r_3+q_1q_2r_3b+r_1r_2r_3y)(cy)\\
	&\hspace{4 mm}+(p_1r_2r_3 + r_1p_2r_3+ q_1q_2r_3c+r_1r_2r_3z)(-by+yi+zj)\\
\end{align*}
\begin{align*}
&= \left(\begin{array}{l}
	p_1p_2p_3 - q_1q_2p_3cz+ q_1r_2p_3cy + r_1q_2p_3cy - r_1r_2p_3by\\
	-p_1q_2q_3cz - q_1p_2q_3cz-q_1q_2q_3bcz-r_1r_2q_3cyz\\
	+p_1r_2q_3cy + r_1p_2q_3cy+ q_1q_2q_3c^2y+r_1r_2q_3cyz\\
	+p_1q_2r_3cy + q_1p_2r_3cy+q_1q_2r_3bcy+r_1r_2r_3cy^2\\
	-p_1r_2r_3by - r_1p_2r_3by- q_1q_2r_3bcy-r_1r_2r_3byz
	\end{array}\right)\\
	&\hspace{4 mm}+ \left(\begin{array}{l}
	p_1q_2p_3 + q_1p_2p_3+q_1q_2p_3b+r_1r_2p_3y\\
	+p_1p_2q_3 - q_1q_2q_3cz+ q_1r_2q_3cy + r_1q_2q_3cy - r_1r_2q_3by\\
	p_1q_2q_3b + q_1p_2q_3b+q_1q_2q_3b^2+r_1r_2q_3by\\
	p_1r_2r_3y + r_1p_2r_3y+ q_1q_2r_3cy+r_1r_2r_3yz
	\end{array}\right)i\\
	&\hspace{4 mm}+ \left(\begin{array}{l}
	p_1r_2p_3 + r_1p_2p_3+ q_1q_2p_3c+r_1r_2p_3z\\
	+p_1p_2r_3 - q_1q_2r_3cz+ q_1r_2r_3cy + r_1q_2r_3cy - r_1r_2r_3by\\
	p_1q_2q_3c + q_1p_2q_3c+q_1q_2q_3bc+r_1r_2q_3cy\\
	p_1r_2r_3z + r_1p_2r_3z+ q_1q_2r_3cz+r_1r_2r_3z^2
	\end{array}\right)j\\
\end{align*}
\begin{align*}
&= \left(\begin{array}{l}
	p_1p_2p_3 - q_1q_2p_3cz+ q_1r_2p_3cy + r_1q_2p_3cy - r_1r_2p_3by\\
	-p_1q_2q_3cz - q_1p_2q_3cz-q_1q_2q_3bcz\\
	+p_1r_2q_3cy + r_1p_2q_3cy+ q_1q_2q_3c^2y\\
	+p_1q_2r_3cy + q_1p_2r_3cy+r_1r_2r_3cy^2\\
	-p_1r_2r_3by - r_1p_2r_3by-r_1r_2r_3byz
	\end{array}\right)\\
	&\hspace{4 mm}+ \left(\begin{array}{l}
	p_1q_2p_3 + q_1p_2p_3+q_1q_2p_3b+r_1r_2p_3y\\
	+p_1p_2q_3 - q_1q_2q_3cz+ q_1r_2q_3cy + r_1q_2q_3cy \\
	p_1q_2q_3b + q_1p_2q_3b+q_1q_2q_3b^2\\
	p_1r_2r_3y + r_1p_2r_3y+ q_1q_2r_3cy+r_1r_2r_3yz
	\end{array}\right)i\\
	&\hspace{4 mm}+ \left(\begin{array}{l}
	p_1r_2p_3 + r_1p_2p_3+ q_1q_2p_3c+r_1r_2p_3z\\
	+p_1p_2r_3+ q_1r_2r_3cy + r_1q_2r_3cy - r_1r_2r_3by\\
	p_1q_2q_3c + q_1p_2q_3c+q_1q_2q_3bc+r_1r_2q_3cy\\
	p_1r_2r_3z + r_1p_2r_3z+r_1r_2r_3z^2
	\end{array}\right)j
\end{align*} 

Next we compute $x_1(x_2x_3)$:
\begin{align*}
(p_1 + q_1i + r_1j)[(p_2 + q_2i + r_2j)(p_3 + q_3i + r_3j)] \ \ \ \ \ \ \ \ \ \ \ \ \ \ \ \ \ \ \ \ \ \ \ \ \ \ \ \ \ \ \ \ \ \ \ 
\end{align*}
\vspace{-16 mm}
\begin{align*}
&=  (p_1 + q_1i + r_1j)\left( \begin{array}{l}p_2p_3 + p_2q_3i +p_2r_3j \\  
					+ q_2p_3i + q_2q_3i^2 + q_2r_3ij\\
					+ r_2p_3j + r_2q_3ji + r_2r_3j^2\\
\end{array} \right)\\
\end{align*}
\begin{align*}
&= (p_1 + q_1i + r_1j)\left( \begin{array}{l}
	p_2p_3 + (p_2q_3 + q_2p_3)i + (p_2r_3+r_2p_3)j\\
	+ q_2q_3(-cz+bi+cj)\\
	+q_2r_3(cy)\\
	+r_2q_3(cy)\\
	+r_2r_3(-by+yi+zj)
\end{array} \right)\\
&= p_1 \left( \begin{array}{l}
	p_2p_3  - q_2q_3cz + q_2r_3cy + r_2q_3cy - r_2r_3by\\
	+(p_2q_3 + q_2p_3+q_2q_3b+r_2r_3y)i\\
	+(p_2r_3+r_2p_3+q_2q_3c+r_2r_3z)j
	\end{array}\right)\\
	&\hspace{4 mm}+q_1 i\left( \begin{array}{l}
	p_2p_3  - q_2q_3cz + q_2r_3cy + r_2q_3cy - r_2r_3by\\
	+(p_2q_3 + q_2p_3+q_2q_3b+r_2r_3y)i\\
	+(p_2r_3+r_2p_3+q_2q_3c+r_2r_3z)j
	\end{array}\right)\\
	&\hspace{4 mm}+r_1j\left( \begin{array}{l}
	p_2p_3  - q_2q_3cz + q_2r_3cy + r_2q_3cy - r_2r_3by\\
	+(p_2q_3 + q_2p_3+q_2q_3b+r_2r_3y)i\\
	+(p_2r_3+r_2p_3+q_2q_3c+r_2r_3z)j
	\end{array}\right)\\
&= \left( \begin{array}{l}
	p_1p_2p_3  - p_1q_2q_3cz + p_1q_2r_3cy + p_1r_2q_3cy - p_1r_2r_3by\\
	+(p_1p_2q_3 + p_1q_2p_3+p_1q_2q_3b+p_1r_2r_3y)i\\
	+(p_1p_2r_3+p_1r_2p_3+p_1q_2q_3c+p_1r_2r_3z)j
	\end{array}\right)\\
	&\hspace{4 mm}+\left( \begin{array}{l}
	(q_1p_2p_3  - q_1q_2q_3cz + q_1q_2r_3cy + q_1r_2q_3cy - q_1r_2r_3by)i\\
	+(q_1p_2q_3 + q_1q_2p_3+q_1q_2q_3b+q_1r_2r_3y)i^2\\
	+(q_1p_2r_3+q_1r_2p_3+q_1q_2q_3c+q_1r_2r_3z)ij
	\end{array}\right)\\
	&\hspace{4 mm}+\left( \begin{array}{l}
	(r_1p_2p_3  - r_1q_2q_3cz + r_1q_2r_3cy + r_1r_2q_3cy - r_1r_2r_3by)j\\
	+(r_1p_2q_3 + r_1q_2p_3+r_1q_2q_3b+r_1r_2r_3y)ji\\
	+(r_1p_2r_3+r_1r_2p_3+r_1q_2q_3c+r_1r_2r_3z)j^2
	\end{array}\right)\\
\end{align*}
\begin{align*}
&= 	p_1p_2p_3  - p_1q_2q_3cz + p_1q_2r_3cy + p_1r_2q_3cy - p_1r_2r_3by\\
	&\hspace{4 mm}+ \left( \begin{array}{l}
	p_1p_2q_3 + p_1q_2p_3+p_1q_2q_3b+p_1r_2r_3y\\
	+q_1p_2p_3  - q_1q_2q_3cz + q_1q_2r_3cy + q_1r_2q_3cy - q_1r_2r_3by
	\end{array}\right)i\\
	&\hspace{4 mm}+ \left( \begin{array}{l}
	p_1p_2r_3+p_1r_2p_3+p_1q_2q_3c+p_1r_2r_3z\\
	+r_1p_2p_3  - r_1q_2q_3cz + r_1q_2r_3cy + r_1r_2q_3cy - r_1r_2r_3by
	\end{array}\right)j\\
	&\hspace{4 mm}+(q_1p_2q_3 + q_1q_2p_3+q_1q_2q_3b+q_1r_2r_3y)(-cz+bi+cj)\\
	&\hspace{4 mm}+(q_1p_2r_3+q_1r_2p_3+q_1q_2q_3c+q_1r_2r_3z)(cy)\\
	&\hspace{4 mm}+(r_1p_2q_3 + r_1q_2p_3+r_1q_2q_3b+r_1r_2r_3y)(cy)\\
	&\hspace{4 mm}+(r_1p_2r_3+r_1r_2p_3+r_1q_2q_3c+r_1r_2r_3z)(-by+yi+zj)\\
&= \left(\begin{array}{l}  
	p_1p_2p_3  - p_1q_2q_3cz + p_1q_2r_3cy + p_1r_2q_3cy - p_1r_2r_3by\\
	-q_1p_2q_3cz - q_1q_2p_3cz - q_1q_2q_3bcz - q_1r_2r_3cyz\\
	+q_1p_2r_3cy+q_1r_2p_3cy+q_1q_2q_3c^2y+q_1r_2r_3cyz\\
	+r_1p_2q_3cy + r_1q_2p_3cy+r_1q_2q_3bcy+r_1r_2r_3cy^2\\
	-r_1p_2r_3by-r_1r_2p_3by-r_1q_2q_3bcy-r_1r_2r_3byz
	\end{array}\right)\\
	&\hspace{4 mm}+\left( \begin{array}{l}
	p_1p_2q_3 + p_1q_2p_3+p_1q_2q_3b+p_1r_2r_3y\\
	+q_1p_2p_3  - q_1q_2q_3cz + q_1q_2r_3cy + q_1r_2q_3cy - q_1r_2r_3by\\
	+q_1p_2q_3b + q_1q_2p_3b+q_1q_2q_3b^2+q_1r_2r_3by\\
	+r_1p_2r_3y+r_1r_2p_3y+r_1q_2q_3cy+r_1r_2r_3yz
	\end{array}\right)i\\
	&\hspace{4 mm}+\left( \begin{array}{l}
	p_1p_2r_3+p_1r_2p_3+p_1q_2q_3c+p_1r_2r_3z\\
	+r_1p_2p_3  - r_1q_2q_3cz + r_1q_2r_3cy + r_1r_2q_3cy - r_1r_2r_3by\\
	+q_1p_2q_3c + q_1q_2p_3c+q_1q_2q_3bc+q_1r_2r_3cy\\
	+r_1p_2r_3z+r_1r_2p_3z+r_1q_2q_3cz+r_1r_2r_3z^2
	\end{array}\right)j\\
\end{align*}
\begin{align*}
&= \left(\begin{array}{l}  
	p_1p_2p_3  - p_1q_2q_3cz + p_1q_2r_3cy + p_1r_2q_3cy - p_1r_2r_3by\\
	-q_1p_2q_3cz - q_1q_2p_3cz - q_1q_2q_3bcz \\
	+q_1p_2r_3cy+q_1r_2p_3cy+q_1q_2q_3c^2y\\
	+r_1p_2q_3cy + r_1q_2p_3cy+r_1r_2r_3cy^2\\
	-r_1p_2r_3by-r_1r_2p_3by-r_1r_2r_3byz
	\end{array}\right)\\
	&\hspace{4 mm}+\left( \begin{array}{l}
	p_1p_2q_3 + p_1q_2p_3+p_1q_2q_3b+p_1r_2r_3y\\
	+q_1p_2p_3  - q_1q_2q_3cz + q_1q_2r_3cy + q_1r_2q_3cy \\
	+q_1p_2q_3b + q_1q_2p_3b+q_1q_2q_3b^2\\
	+r_1p_2r_3y+r_1r_2p_3y+r_1q_2q_3cy+r_1r_2r_3yz
	\end{array}\right)i\\
	&\hspace{4 mm}+\left( \begin{array}{l}
	p_1p_2r_3+p_1r_2p_3+p_1q_2q_3c+p_1r_2r_3z\\
	+r_1p_2p_3  + r_1q_2r_3cy + r_1r_2q_3cy - r_1r_2r_3by\\
	+q_1p_2q_3c + q_1q_2p_3c+q_1q_2q_3bc+q_1r_2r_3cy\\
	+r_1p_2r_3z+r_1r_2p_3z+r_1r_2r_3z^2
	\end{array}\right)j
\end{align*}

Observe that in case (C), $(x_1x_2)x_3 = (x_1x_2)x_3$.  Next consider case (E), in which the multiplication table for $A$ is
\begin{equation*}\label{eqn:exceptional}
\begin{array}{ccl}
i^2 & = & ni\\
ij&=&0\\
ji&=&-mn + mi + nj\\
j^2&=&mj
\end{array}
\end{equation*}
where $m, n$ are arbitrary elements of $R$ (with $m, n$ not both 0).\\

We compute $(x_1x_2)x_3$:
\begin{align*}
[(p_1 + q_1i + r_1j)(p_2 + q_2i + r_2j)](p_3 + q_3i + r_3j) \ \ \ \ \ \ \ \ \ \ \ \ \ \ \ \ \ \ \ \ \ \ \ \ \ \ \ \ \ \ \ \ \ \ \ 
\end{align*}
\vspace{-16 mm}
\begin{align*}
&= \left( \begin{array}{l}p_1p_2 + p_1q_2i +p_1r_2j \\  
					+ q_1p_2i + q_1q_2i^2 + q_1r_2ij\\
					+ r_1p_2j + r_1q_2ji + r_1r_2j^2\\
	\end{array} \right) (p_3 + q_3i + r_3j)\\
&= \left( \begin{array}{l}p_1p_2 + (p_1q_2+q_1p_2)i +(p_1r_2+r_1p_2)j \\ 
	+q_1q_2(ni)\\
	+q_1r_2(0)\\
	r_1q_2(-mn + mi + nj)\\
	r_1r_2(mj)
	\end{array} \right) (p_3 + q_3i + r_3j)\\
&= \left( \begin{array}{l}
	p_1p_2 - r_1q_2mn\\
	+(p_1q_2+q_1p_2 + q_1q_2n+ r_1q_2m)i\\
	+(p_1r_2+r_1p_2+r_1q_2n+r_1r_2m)j
	\end{array} \right)p_3\\
	&\hspace{4 mm}+  \left( \begin{array}{l}p_1p_2 - r_1q_2mn\\
	+(p_1q_2+q_1p_2 + q_1q_2n+ r_1q_2m)i\\
	+(p_1r_2+r_1p_2+r_1q_2n+r_1r_2m)j
	\end{array} \right)q_3i\\
	&\hspace{4 mm}+  \left( \begin{array}{l}p_1p_2 - r_1q_2mn\\
	+(p_1q_2+q_1p_2 + q_1q_2n+ r_1q_2m)i\\
	+(p_1r_2+r_1p_2+r_1q_2n+r_1r_2m)j
	\end{array} \right)r_3j\\
\end{align*}
\begin{align*}
&= \left( \begin{array}{l}
	p_1p_2p_3 - r_1q_2p_3mn\\
	+(p_1q_2p_3+q_1p_2p_3 + q_1q_2p_3n+ r_1q_2p_3m)i\\
	+(p_1r_2p_3+r_1p_2p_3+r_1q_2p_3n+r_1r_2p_3m)j
	\end{array} \right)\\
	&\hspace{4 mm}+  \left( \begin{array}{l}(p_1p_2q_3 - r_1q_2q_3mn)i\\
	+(p_1q_2q_3+q_1p_2q_3 + q_1q_2q_3n+ r_1q_2q_3m)i^2\\
	+(p_1r_2q_3+r_1p_2q_3+r_1q_2q_3n+r_1r_2q_3m)ji
	\end{array} \right)\\
	&\hspace{4 mm}+  \left( \begin{array}{l}(p_1p_2r_3 - r_1q_2r_3mn)j\\
	+(p_1q_2r_3+q_1p_2r_3 + q_1q_2r_3n+ r_1q_2r_3m)ij\\
	+(p_1r_2r_3+r_1p_2r_3+r_1q_2r_3n+r_1r_2r_3m)j^2
	\end{array} \right)\\
&= 	p_1p_2p_3 - r_1q_2p_3mn\\
	&\hspace{4 mm}+\left(\begin{array}{l} 
	p_1q_2p_3+q_1p_2p_3 + q_1q_2p_3n+ r_1q_2p_3m\\
	+p_1p_2q_3 - r_1q_2q_3mn \end{array}\right)i\\
	&\hspace{4 mm}+\left(\begin{array}{l} 
	p_1r_2p_3+r_1p_2p_3+r_1q_2p_3n+r_1r_2p_3m\\
	+p_1p_2r_3 - r_1q_2r_3mn
	\end{array}\right)j\\
	&\hspace{4 mm}+(p_1q_2q_3+q_1p_2q_3 + q_1q_2q_3n+ r_1q_2q_3m)(ni)\\
	&\hspace{4 mm}+(p_1r_2q_3+r_1p_2q_3+r_1q_2q_3n+r_1r_2q_3m)(-mn+mi+nj)\\
	&\hspace{4 mm}+(p_1q_2r_3+q_1p_2r_3 + q_1q_2r_3n+ r_1q_2r_3m)(0)\\
	&\hspace{4 mm}+(p_1r_2r_3+r_1p_2r_3+r_1q_2r_3n+r_1r_2r_3m)(mj)\\
\end{align*}
\begin{align*}
&= \left( \begin{array}{l} 
	p_1p_2p_3 - r_1q_2p_3mn\\
	-p_1r_2q_3mn-r_1p_2q_3mn-r_1q_2q_3mn^2-r_1r_2q_3m^2n
	 \end{array}\right)\\
	 &\hspace{4 mm}+ \left( \begin{array}{l}
	 p_1q_2p_3+q_1p_2p_3 + q_1q_2p_3n+ r_1q_2p_3m\\
	+p_1p_2q_3 - r_1q_2q_3mn\\
	+p_1q_2q_3n+q_1p_2q_3n + q_1q_2q_3n^2+ r_1q_2q_3mn\\
	+p_1r_2q_3m+r_1p_2q_3m+r_1q_2q_3mn+r_1r_2q_3m^2\\
	 \end{array}\right)i\\
	 &\hspace{4 mm}+ \left( \begin{array}{l}
	p_1r_2p_3+r_1p_2p_3+r_1q_2p_3n+r_1r_2p_3m\\
	+p_1p_2r_3 - r_1q_2r_3mn\\
	+p_1r_2q_3n+r_1p_2q_3n+r_1q_2q_3n^2+r_1r_2q_3mn\\
	p_1r_2r_3m+r_1p_2r_3m+r_1q_2r_3mn+r_1r_2r_3m^2
	 \end{array}\right)j\\
&= \left( \begin{array}{l} 
	p_1p_2p_3 - r_1q_2p_3mn\\
	-p_1r_2q_3mn-r_1p_2q_3mn-r_1q_2q_3mn^2-r_1r_2q_3m^2n
	 \end{array}\right)\\
	 &\hspace{4 mm}+ \left( \begin{array}{l}
	 p_1q_2p_3+q_1p_2p_3 + q_1q_2p_3n+ r_1q_2p_3m\\
	+p_1p_2q_3 \\
	+p_1q_2q_3n+q_1p_2q_3n + q_1q_2q_3n^2\\
	+p_1r_2q_3m+r_1p_2q_3m+r_1q_2q_3mn+r_1r_2q_3m^2\\
	 \end{array}\right)i\\
	 &\hspace{4 mm}+ \left( \begin{array}{l}
	p_1r_2p_3+r_1p_2p_3+r_1q_2p_3n+r_1r_2p_3m\\
	+p_1p_2r_3\\
	+p_1r_2q_3n+r_1p_2q_3n+r_1q_2q_3n^2+r_1r_2q_3mn\\
	p_1r_2r_3m+r_1p_2r_3m+r_1r_2r_3m^2
	 \end{array}\right)j\\
\end{align*}
Next we compute $x_1(x_2x_3)$:
\begin{align*}
(p_1 + q_1i + r_1j)[(p_2 + q_2i + r_2j)(p_3 + q_3i + r_3j)] \ \ \ \ \ \ \ \ \ \ \ \ \ \ \ \ \ \ \ \ \ \ \ \ \ \ \ \ \ \ \ \ \ \ \ 
\end{align*}
\begin{align*}
&=  (p_1 + q_1i + r_1j)
	\left( \begin{array}{l}p_2p_3 + p_2q_3i +p_2r_3j \\  
	+ q_2p_3i + q_2q_3i^2 + q_2r_3ij\\
	+ r_2p_3j + r_2q_3ji + r_2r_3j^2\\
	\end{array} \right)\\
&= (p_1 + q_1i + r_1j)\left( \begin{array}{l}
	p_2p_3 + (p_2q_3+q_2p_3)i + (p_2r_3+r_2p_3)j\\
	+q_2q_3(ni)\\
	+q_2r_3(0)\\
	+r_2q_3(-mn+mi+nj)\\
	+r_2r_3(mj)
	\end{array}\right)\\
&= p_1 \left( \begin{array}{l}
	p_2p_3 - r_2q_3mn\\
	+(p_2q_3+q_2p_3+q_2q_3n+r_2q_3m)i\\
	+(p_2r_3+r_2p_3+r_2q_3n+ r_2r_3m)j
	\end{array}\right)\\
	&\hspace{4 mm}+ q_1i\left( \begin{array}{l}
	p_2p_3 - r_2q_3mn\\
	+(p_2q_3+q_2p_3+q_2q_3n+r_2q_3m)i\\
	+(p_2r_3+r_2p_3+r_2q_3n+ r_2r_3m)j	
	\end{array}\right)\\
	&\hspace{4 mm}+ r_1j\left( \begin{array}{l}
	p_2p_3 - r_2q_3mn\\
	+(p_2q_3+q_2p_3+q_2q_3n+r_2q_3m)i\\
	+(p_2r_3+r_2p_3+r_2q_3n+ r_2r_3m)j	
	\end{array}\right)\\
\end{align*}
\begin{align*}
&=  \left( \begin{array}{l}
	p_1p_2p_3 - p_1r_2q_3mn\\
	+(p_1p_2q_3+p_1q_2p_3+p_1q_2q_3n+p_1r_2q_3m)i\\
	+(p_1p_2r_3+p_1r_2p_3+p_1r_2q_3n+ p_1r_2r_3m)j
	\end{array}\right)\\
	&\hspace{4 mm}+ \left( \begin{array}{l}
	(q_1p_2p_3 - q_1r_2q_3mn)i\\
	+(q_1p_2q_3+q_1q_2p_3+q_1q_2q_3n+q_1r_2q_3m)i^2\\
	+(q_1p_2r_3+q_1r_2p_3+q_1r_2q_3n+ q_1r_2r_3m)ij	
	\end{array}\right)\\
	&\hspace{4 mm}+\left( \begin{array}{l}
	(r_1p_2p_3 - r_1r_2q_3mn)j\\
	+(r_1p_2q_3+r_1q_2p_3+r_1q_2q_3n+r_1r_2q_3m)ji\\
	+(r_1p_2r_3+r_1r_2p_3+r_1r_2q_3n+ r_1r_2r_3m)j^2	
	\end{array}\right)\\
&= 	p_1p_2p_3 - p_1r_2q_3mn\\
	&\hspace{4 mm}+\left( \begin{array}{l}
	p_1p_2q_3+p_1q_2p_3+p_1q_2q_3n+p_1r_2q_3m\\
	+q_1p_2p_3 - q_1r_2q_3mn
	\end{array}\right)i\\
	&\hspace{4 mm}+\left( \begin{array}{l}
	p_1p_2r_3+p_1r_2p_3+p_1r_2q_3n+ p_1r_2r_3m\\
	+r_1p_2p_3 - r_1r_2q_3mn
	\end{array}\right)j\\
	&\hspace{4 mm}+(q_1p_2q_3+q_1q_2p_3+q_1q_2q_3n+q_1r_2q_3m)(ni)\\
	&\hspace{4 mm}+(q_1p_2r_3+q_1r_2p_3+q_1r_2q_3n+ q_1r_2r_3m)(0)\\
	&\hspace{4 mm}+(r_1p_2q_3+r_1q_2p_3+r_1q_2q_3n+r_1r_2q_3m)(-mn + mi + nj)\\
	&\hspace{4 mm}+(r_1p_2r_3+r_1r_2p_3+r_1r_2q_3n+ r_1r_2r_3m)(mj)\\
\end{align*}
\begin{align*}
&= \left( \begin{array}{l}
	p_1p_2p_3 - p_1r_2q_3mn\\
	-r_1p_2q_3mn-r_1q_2p_3mn-r_1q_2q_3mn^2-r_1r_2q_3m^2n
	\end{array}\right)\\
	&\hspace{4 mm}+\left( \begin{array}{l} 
	p_1p_2q_3+p_1q_2p_3+p_1q_2q_3n+p_1r_2q_3m\\
	+q_1p_2p_3 - q_1r_2q_3mn\\
	+q_1p_2q_3n+q_1q_2p_3n+q_1q_2q_3n^2+q_1r_2q_3mn\\
	+r_1p_2q_3m+r_1q_2p_3m+r_1q_2q_3mn+r_1r_2q_3m^2
	\end{array}\right)i\\
	&\hspace{4 mm}+\left(  \begin{array}{l} 
	p_1p_2r_3+p_1r_2p_3+p_1r_2q_3n+ p_1r_2r_3m\\
	+r_1p_2p_3 - r_1r_2q_3mn\\
	+r_1p_2q_3n+r_1q_2p_3n+r_1q_2q_3n^2+r_1r_2q_3mn\\
	+r_1p_2r_3m+r_1r_2p_3m+r_1r_2q_3mn+ r_1r_2r_3m^2
	\end{array}\right)j\\
&= \left( \begin{array}{l}
	p_1p_2p_3 - p_1r_2q_3mn\\
	-r_1p_2q_3mn-r_1q_2p_3mn-r_1q_2q_3mn^2-r_1r_2q_3m^2n
	\end{array}\right)\\
	&\hspace{4 mm}+\left( \begin{array}{l} 
	p_1p_2q_3+p_1q_2p_3+p_1q_2q_3n+p_1r_2q_3m\\
	+q_1p_2p_3\\
	+q_1p_2q_3n+q_1q_2p_3n+q_1q_2q_3n^2\\
	+r_1p_2q_3m+r_1q_2p_3m+r_1q_2q_3mn+r_1r_2q_3m^2
	\end{array}\right)i\\
	&\hspace{4 mm}+\left(  \begin{array}{l} 
	p_1p_2r_3+p_1r_2p_3+p_1r_2q_3n+ p_1r_2r_3m\\
	+r_1p_2p_3 \\
	+r_1p_2q_3n+r_1q_2p_3n+r_1q_2q_3n^2\\
	+r_1p_2r_3m+r_1r_2p_3m+r_1r_2q_3mn+ r_1r_2r_3m^2
	\end{array}\right)j
\end{align*}
Observe that in case (E), $(x_1x_2)x_3 = x_1(x_2x_3)$. Since associativity holds in both cases, the multiplication on $M$ defined by table \ref{eqn:mtable} and relations \ref{eqn:relations}$.10$ is associative.

\end{document}